\documentclass[10pt,draft]{article}
\usepackage{amsmath,amscd}
\usepackage{amssymb,latexsym,amsthm}
\usepackage{color}
\usepackage[spanish,english]{babel}
\usepackage{amssymb}
\usepackage{color}
\usepackage{amsmath,amsthm,amscd}
\usepackage[latin1]{inputenc}
\usepackage{lscape}
\usepackage{fancyhdr}
\usepackage{amsfonts}
\usepackage{pb-diagram}

\numberwithin{equation}{section}
\newtheorem{theorem}{Theorem}[section]

\newtheorem{definition}[theorem]{Definition}
\newtheorem{proposition}[theorem]{Proposition}
\newtheorem{lemma}[theorem]{Lemma}

\newtheorem{corollary}[theorem]{Corollary}

\theoremstyle{definition}

\newtheorem{example}[theorem]{Example}

\newtheorem{remark}[theorem]{Remark}

\newcommand{\cA}{\mbox{${\cal A}$}}

\newcommand{\cO}{\mbox{${\cal O}$}}

\newcommand{\cU}{\mbox{${\cal U}$}}

\newcommand{\cW}{\mbox{${\cal W}$}}

\title{\textbf{Some homological properties of skew $PBW$ extensions}}
\author{Oswaldo Lezama\\
\texttt{jolezamas@unal.edu.co}
\\Armando Reyes\\ Seminario de Álgebra Constructiva - SAC$^2$\\ Departamento de Matemáticas\\ Universidad Nacional de
Colombia, Sede Bogot\'a}
\date{}
\begin{document}
\maketitle
\begin{abstract}
\noindent We prove that if $R$ is a left Noetherian and left regular ring then the same is true for
any bijective skew $PBW$ extension $A$ of $R$. From this we get Serre's Theorem for such
extensions. We show that skew $PBW$ extensions and its localizations include a wide variety of
rings and algebras of interest for modern mathematical physics such as $PBW$ extensions, well known
classes of Ore algebras, operator algebras, diffusion algebras, quantum algebras, quadratic
algebras in 3-variables, skew quantum polynomials, among many others. We estimate the global, Krull
and Goldie dimensions, and also Quillen's $K$-groups.

\bigskip

\noindent \textit{Key words and phrases.} Noetherian regular non-commutative rings, graded rings,
$PBW$ extensions, quantum algebras, Serre's Theorem, global, Krull and Goldie dimensions, algebraic
$K$-theory.

\bigskip

\noindent 2010 \textit{Mathematics Subject Classification.} Primary: 16S80, 16W25, 16S36, 16U20,
18F25. Secondary: 16W50, 16E65.
\end{abstract}
\section{Definitions and elementary properties}\label{definitionexamplesspbw}
In this section we recall the definition of skew $PBW$ (Poincaré-Birkhoff-Witt) extensions defined
firstly in \cite{LezamaGallego}, and also some elementary properties about the polynomial
interpretation of this kind of non-commutative rings. Skew $PBW$ extensions are a generalization of
$PBW$ extensions introduced in \cite{Bell}, and include a wide variety of rings and algebras of
interest for modern mathematical physics as we will show in Section \ref{Examplestrad}.
\begin{definition}\label{gpbwextension}
Let $R$ and $A$ be rings. We say that $A$ is an \textit{skew $PBW$ extension of $R$} $($also called
a $\sigma-PBW$ extension of $R$$)$ if the following conditions hold:
\begin{enumerate}
\item[\rm (i)]$R\subseteq A$.
\item[\rm (ii)]There exist finite elements $x_1,\dots ,x_n\in A$ such $A$ is a left $R$-free module with basis
\begin{center}
${\rm Mon}(A):= \{x^{\alpha}=x_1^{\alpha_1}\cdots
x_n^{\alpha_n}\mid \alpha=(\alpha_1,\dots ,\alpha_n)\in
\mathbb{N}^n\}$.
\end{center}
In this case we say also that \textit{$A$ is a left polynomial
ring over $R$} with respect to $\{x_1,\dots,x_n\}$ and ${\rm Mon}(A)$ is
the set of standard monomials of $A$. In addition, $x_1^0\cdots
x_n^0:=1\in {\rm Mon}(A)$.
\item[\rm (iii)]For every $1\leq i\leq n$ and $r\in R-\{0\}$ there exists $c_{i,r}\in R-\{0\}$ such that
\begin{equation}\label{sigmadefinicion1}
x_ir-c_{i,r}x_i\in R.
\end{equation}
\item[\rm (iv)]For every $1\leq i,j\leq n$ there exists $c_{i,j}\in R-\{0\}$ such that
\begin{equation}\label{sigmadefinicion2}
x_jx_i-c_{i,j}x_ix_j\in R+Rx_1+\cdots +Rx_n.
\end{equation}
Under these conditions we will write $A:=\sigma(R)\langle
x_1,\dots ,x_n\rangle$.
\end{enumerate}
\end{definition}
The following proposition justifies the notation and the
alternative name given for the skew $PBW$ extensions.
\begin{proposition}\label{sigmadefinition}
Let $A$ be an skew $PBW$ extension of $R$. Then, for every $1\leq i\leq n$, there exists an
injective ring endomorphism $\sigma_i:R\rightarrow R$ and a $\sigma_i$-derivation
$\delta_i:R\rightarrow R$ such that
\begin{center}
$x_ir=\sigma_i(r)x_i+\delta_i(r)$,
\end{center}
for each $r\in R$.
\end{proposition}
\begin{proof}
We present the proof given in \cite{LezamaGallego} Proposition 3:
for every $1\leq i\leq n$ and each $r\in R$ we have elements
$c_{i,r},r_i\in R$ such that $x_ir=c_{i,r}x_i+r_i$; since ${\rm Mon}(A)$
is a $R$-basis of $A$ then $c_{i,r}$ and $r_i$ are unique for $r$,
so we define $\sigma_i,\delta_i:R\rightarrow R$ by
$\sigma_i(r):=c_{i,r}$, $\delta_i(r):=r_i$. It is easy to check
that $\sigma_i$ is a ring endomorphism and $\delta_i$ is a
$\sigma_i$-derivation of $R$, i.e.,
$\delta_i(r+r')=\delta_i(r)+\delta_i(r')$ and
$\delta_i(rr')=\sigma_i(r)\delta_i(r')+\delta_i(r)r'$, for any
$r,r'\in R$. Moreover, by Definition \ref{gpbwextension} (iii),
$c_{i,r}\neq 0$ for $r\neq 0$. This means that $\sigma_i$ is
injective.
\end{proof}
A particular case of skew $PBW$ extension is when all derivations $\delta_i$ are zero. Another
interesting case is when all $\sigma_i$ are bijective and the constants $c_{ij}$ are invertible. We
recall the following definition (cf. \cite{LezamaGallego}).
\begin{definition}\label{sigmapbwderivationtype}
Let $A$ be a skew $PBW$ extension.
\begin{enumerate}
\item[\rm (a)]
$A$ is quasi-commutative if the conditions {\rm(}iii{\rm)} and
{\rm(}iv{\rm)} in Definition \ref{gpbwextension} are replaced by
\begin{enumerate}
\item[\rm (iii')]For every $1\leq i\leq n$ and $r\in R-\{0\}$ there exists $c_{i,r}\in R-\{0\}$ such that
\begin{equation}
x_ir=c_{i,r}x_i.
\end{equation}
\item[\rm (iv')]For every $1\leq i,j\leq n$ there exists $c_{i,j}\in R-\{0\}$ such that
\begin{equation}
x_jx_i=c_{i,j}x_ix_j.
\end{equation}
\end{enumerate}
\item[\rm (b)]$A$ is bijective if $\sigma_i$ is bijective for
every $1\leq i\leq n$ and $c_{i,j}$ is invertible for any $1\leq
i<j\leq n$.
\end{enumerate}
\end{definition}
\begin{definition}
Let $A$ be an skew $PBW$ extension of $R$ with endomorphisms $\sigma_i$, $1\leq i\leq n$, as in
Proposition \ref{sigmadefinition}.
\begin{enumerate}
\item[\rm (i)]For $\alpha=(\alpha_1,\dots,\alpha_n)\in \mathbb{N}^n$,
$\sigma^{\alpha}:=\sigma_1^{\alpha_1}\cdots \sigma_n^{\alpha_n}$,
$|\alpha|:=\alpha_1+\cdots+\alpha_n$. If
$\beta=(\beta_1,\dots,\beta_n)\in \mathbb{N}^n$, then
$\alpha+\beta:=(\alpha_1+\beta_1,\dots,\alpha_n+\beta_n)$.
\item[\rm (ii)]For $X=x^{\alpha}\in {\rm Mon}(A)$,
$\exp(X):=\alpha$ and $\deg(X):=|\alpha|$.
\item[\rm (iii)]If $f=c_1X_1+\cdots +c_tX_t$,
with $X_i\in Mon(A)$ and $c_i\in R-\{0\}$, then
$\deg(f):=\max\{\deg(X_i)\}_{i=1}^t.$
\end{enumerate}
\end{definition}
The skew $PBW$ extensions can be characterized in a similar way as
was done in \cite{Gomez-Torrecillas} for $PBW$ rings.
\begin{theorem}\label{coefficientes}
Let $A$ be a left polynomial ring over $R$ w.r.t. $\{x_1,\dots,x_n\}$. $A$ is an skew $PBW$
extension of $R$ if and only if the following conditions hold:
\begin{enumerate}
\item[\rm (a)]For every $x^{\alpha}\in {\rm Mon}(A)$ and every $0\neq
r\in R$ there exist unique elements
$r_{\alpha}:=\sigma^{\alpha}(r)\in R-\{0\}$ and $p_{\alpha ,r}\in
A$ such that
\begin{equation}\label{611}
x^{\alpha}r=r_{\alpha}x^{\alpha}+p_{\alpha , r},
\end{equation}
where $p_{\alpha ,r}=0$ or $\deg(p_{\alpha ,r})<|\alpha|$ if
$p_{\alpha , r}\neq 0$. Moreover, if $r$ is left invertible, then
$r_\alpha$ is left invertible.

\item[\rm (b)]For every $x^{\alpha},x^{\beta}\in {\rm Mon}(A)$ there
exist unique elements $c_{\alpha,\beta}\in R$ and
$p_{\alpha,\beta}\in A$ such that
\begin{equation}\label{612}
x^{\alpha}x^{\beta}=c_{\alpha,\beta}x^{\alpha+\beta}+p_{\alpha,\beta},
\end{equation}
where $c_{\alpha,\beta}$ is left invertible, $p_{\alpha,\beta}=0$
or $\deg(p_{\alpha,\beta})<|\alpha+\beta|$ if
$p_{\alpha,\beta}\neq 0$.
\end{enumerate}
\begin{proof}
See \cite{LezamaGallego} Theorem 7.
\end{proof}
\end{theorem}
We remember also the following facts from \cite{LezamaGallego}
Remark 8.
\begin{remark}\label{identities}
(i) A left inverse of $c_{\alpha,\beta}$ will be denoted by $c_{\alpha,\beta}'$. We observe that if
$\alpha=0$ or $\beta=0$, then $c_{\alpha,\beta}=1$ and hence $c_{\alpha,\beta}'=1$.

(ii) Let $\theta,\gamma,\beta\in \mathbb{N}^n$ and $c\in R$. Then we have the following identities:
\begin{center}
$\sigma^\theta(c_{\gamma,\beta})c_{\theta,\gamma+\beta}=c_{\theta,\gamma}c_{\theta+\gamma,\beta}$,

$\sigma^\theta(\sigma^\gamma (c))c_{\theta,\gamma}=c_{\theta,\gamma}\sigma^{\theta+\gamma}(c)$.
\end{center}
(iii) We observe that if $A$ is quasi-commutative, then $p_{\alpha,r}=0$ and $p_{\alpha,\beta}=0$
for every $0\neq r\in R$ and every $\alpha,\beta \in \mathbb{N}^n$.

(iv) If $A$ is bijective, then $c_{\alpha,\beta}$ is invertible for any $\alpha,\beta\in
\mathbb{N}^n$.

(v) In $Mon(A)$ we define
\begin{center}
$x^{\alpha}\succeq x^{\beta}\Longleftrightarrow
\begin{cases}
x^{\alpha}=x^{\beta}\\
\text{or} & \\
x^{\alpha}\neq x^{\beta}\, \text{but} \, |\alpha|> |\beta| & \\
\text{or} & \\
x^{\alpha}\neq x^{\beta},|\alpha|=|\beta|\, \text{but $\exists$ $i$ with} &
\alpha_1=\beta_1,\dots,\alpha_{i-1}=\beta_{i-1},\alpha_i>\beta_i.
\end{cases}$
\end{center}
It is clear that this is a total order on $Mon(A)$. If $x^{\alpha}\succeq x^{\beta}$ but
$x^{\alpha}\neq x^{\beta}$, we write $x^{\alpha}\succ x^{\beta}$. Each element $f\in A$ can be
represented in a unique way as $f=c_1x^{\alpha_1}+\cdots +c_tx^{\alpha_t}$, with $c_i\in R-\{0\}$,
$1\leq i\leq t$, and $x^{\alpha_1}\succ \cdots \succ x^{\alpha_t}$. We say that $x^{\alpha_1}$ is
the \textit{leader monomial} of $f$ and we write $lm(f):=x^{\alpha_1}$ ; $c_1$ is the
\textit{leader coefficient} of $f$, $lc(f):=c_1$, and $c_1x^{\alpha_1}$ is the \textit{leader term}
of $f$ denoted by $lt(f):=c_1x^{\alpha_1}$.
 We observe that
\begin{center}
$x^{\alpha}\succ x^{\beta}\Rightarrow
lm(x^{\gamma}x^{\alpha}x^{\lambda})\succ
lm(x^{\gamma}x^{\beta}x^{\lambda})$, for every
$x^{\gamma},x^{\lambda}\in Mon(A)$.
\end{center}
\end{remark}
A natural and useful result that we will use later is the
following property.
\begin{proposition}\label{1.1.10a}
Let $A$ be a bijective skew $PBW$ extension of a ring $R$. Then,
$A_R$ is free with basis ${\rm Mon}(A)$.
\end{proposition}
\begin{proof}
First note that $A_R$ is a module where the product $f\cdot r $ is
defined by the multiplication in $A$: $f\cdot r :=fr$, $f\in A$,
$r\in R$. We prove next that ${\rm Mon}(A)$ is a system of
generators of $A$. Let $f\in A$, then $f$ is a finite summa of
terms like $rx^{\alpha}$, with $r\in R$ and $x^{\alpha}\in {\rm
Mon}(A)$, so it is enough to prove that each of these terms is a
right linear $R$-combination of elements of ${\rm Mon}(A)$. From Theorem
\ref{coefficientes},
$rx^{\alpha}=x^{\alpha}\sigma^{-\alpha}(r)-p_{\alpha,\sigma^{-\alpha}(r)}$,
with $\deg(p_{\alpha,\sigma^{-\alpha}(r)})<|\alpha|$ if
$p_{\alpha,\sigma^{-\alpha}(r)}\neq 0$, so by induction on
$|\alpha|$ we get the result.

Now we will show that $Mon(A)$ is linearly independent: let
$x^{\alpha_1}r_1+\cdots x^{\alpha_t}r_t=0$, with
$x^{\alpha_1}\succ \cdots \succ x^{\alpha_t}$ for the total order
$\succeq$ on $Mon(A)$ defined in the previous remark, then
$\sigma^{\alpha_1}(r_1)x^{\alpha_1}+p_{\alpha_1,r_1}+\cdots
+\sigma^{\alpha_t}(r_t)x^{\alpha_t}+p_{\alpha_t,r_t}=0$, with
$\deg(p_{\alpha_i,r_i})<|\alpha_i|$ if $p_{\alpha_i,r_i}\neq 0$,
$1\leq i\leq t$; hence, $\sigma^{\alpha_1}(r_1)=0$ and from this
$r_1=0$. By induction on $t$ we obtain the result.
\end{proof}
\section{Key theorems} In this section we prove some key results of the paper. We start with the
following proposition that establishes that one can construct a quasi-commutative skew $PBW$
extension from a given skew $PBW$ extension of a ring $R$.
\begin{proposition}\label{associatedpbw}
Let $A$ be an skew $PBW$ extension of $R$. Then, there exists a quasi-commutative skew $PBW$
extension $A^{\sigma}$ of $R$ in $n$ variables $z_1,\dotsc,z_n$ defined by
\[
z_ir=c_{i,r}z_i,\ \ z_jz_i=c_{i,j}z_iz_j,\ \ 1\le i,j\le n,
\]
where $c_{i,r},c_{i,j}$ are the same constants that define $A$. Moreover, if $A$ is bijective then $A^{\sigma}$ is also bijective.
\begin{proof}
We consider $n$ variables $z_1,\dots,z_n$ and the set of standard
monomials $\mathcal{M}:=\{z_1^{\alpha_1}\cdots
z_n^{\alpha_n}\mid \alpha_i\in \mathbb{N}^n ,1\leq i\leq n\}$. Let
$A^{\sigma}$ be the free $R$-module with basis $\mathcal{M}$
(i.e., $A$ and $A^{\sigma}$ are isomorphic $R$-modules). We
define the product in $A^{\sigma}$ by the distributive law and the
rules
\begin{center}
$rz^{\alpha}sz^{\beta}:=r\sigma^{\alpha}(s)c_{\alpha,\beta}z^{\alpha+\beta}$,
\end{center}
where the $\sigma$'s and the constants $c$'s are as in Theorem
\ref{coefficientes}. The identities of Remark \ref{identities}
show that this product is associative, moreover note that
$R\subseteq A^{\sigma}$ since for $r\in R$, $r=rz_1^{0}\cdots
z_n^{0}$. Thus, $A^{\sigma}$ is a quasi-commutative skew $PBW$
extension of $R$, and also, each element $f^{\sigma}$ of
$A^{\sigma}$ corresponds to a unique element $f\in A$, replacing
the variables $x$'s  by the variables $z$'s. The last assertion of
the proposition is obvious.
\end{proof}
\end{proposition}
The first key theorem computes the graduation of a general skew
$PBW$ extension of a ring $R$.
\begin{theorem}\label{1.3.2}
Let $A$ be an arbitrary skew $PBW$ extension of $R$. Then, $A$ is
a filtered ring with filtration given by
\begin{equation}\label{eq1.3.1a}
F_m:=\begin{cases} R & {\rm if}\ \ m=0\\ \{f\in A\mid {\rm deg}(f)\le m\} & {\rm if}\ \ m\ge 1
\end{cases}
\end{equation}
and the corresponding graded ring $Gr(A)$ is a quasi-commutative
skew $PBW$ extension of $R$. Moreover, if $A$ is bijective, then
$Gr(A)$ is a quasi-commutative bijective skew $PBW$ extension of
$R$.
\begin{proof}
Let $A:=\sigma(R)\langle x_1,\dots,x_n\rangle$, with $\sigma_i$ as
in Proposition \ref{sigmadefinition}, and let $c_{ij}$ be the
system of constants as in Definition \ref{sigmapbwderivationtype}.
Let $A^\sigma$ the associated quasi-commutative skew $PBW$
extension defined in Proposition \ref{associatedpbw}.

Associated to filtration (\ref{eq1.3.1a}) we define the graded
ring $Gr(A)$ by
\begin{align*}
Gr(A):=\bigoplus_{m\geq 0}F_{m}/F_{m-1}
\end{align*}
with product given by
\begin{align*}
F_{n}/F_{n-1}\times F_{m}/F_{m-1}&\to F_{n+m}/F_{n+m-1}\\
(a+F_{n-1}, b+F_{m-1})&\mapsto ab+F_{n+m-1}.
\end{align*}
We will prove next that
\begin{equation}\label{eq1.2.1}
Gr(A)\cong A^{\sigma}.
\end{equation}
Let $a\in Gr(A)$, then $a=\oplus_{m\geq 0}\overline{a_{m}}$, where
$\overline{a_{m}}=a_{m}+F_{m-1}$, $a_{m}\in F_{m}$ and
$\overline{a_{m}}\neq \overline{0}$ only for a finite subset of
integers $m$; without loss of generality we can assume that if
$\overline{a_{m}}\neq \overline{0}$ then $a_m$ is an homogeneous
polynomial of degree $m$ presented in the standard form, i.e.,
$a_{m}=\sum_{\alpha\in \mathbb{N}^{n}}c_{\alpha}x^{\alpha}$, where
$|\alpha|=m$ for such $\alpha$ with $c_{\alpha}\neq 0$. Then, we
define
\begin{align*}
\phi: Gr(A)&\to A^{\sigma}\\
\oplus_{m\geq 0}\overline{a_{m}}&\mapsto
\sum_{\overline{a_{m}}\neq \overline{0}}\,a_{m},
\end{align*}

\textit{$\phi$ is well defined}: suppose that $a=\oplus_{m\geq
0}\overline{a_{m}}=\oplus_{m\geq 0}\overline{a_{m}'}$, let $m$
such that $\overline{a_{m}}\neq \overline{0}\neq
\overline{a_{m}'}$, since $\overline{a_{m}}=\overline{a_{m}'}$
then $a_{m}-a_{m}'\in F_{m-1}$; if $a_{m}-a_{m}'\neq 0$ then
$0\leq$ deg$(a_{m}-a_{m'})\leq m-1$ but this is impossible since
$a_{m}$ and $a_{m}'$ are homogeneous of degree $m$, so
$a_{m}=a_{m}'$ and $\sum_{\overline{a_{m}}\neq
\overline{0}}\,a_{m}=\sum_{\overline{a_{m}'}\neq
\overline{0}}\,a_{m}'$.

\textit{$\phi$ is additive}: let $a=\oplus \overline{a_{m}}$,
$b=\oplus \overline{b_{m}}\in Gr(A)$, then $a+b=\oplus
\overline{a_{m}}+\oplus
\overline{b_{m}}=\oplus\overline{a_{m}+b_{m}}$; so
\begin{align*}
\phi(a+b)&=\phi(\oplus\overline{a_{m}+b_{m}})
=\sum_{\overline{a_{m}+b_{m}}\neq \overline{0}}\,a_{m}+b_{m}
=\sum_{\overline{a_{m}}\neq
\overline{0}}\,a_{m}+\sum_{\overline{b_{m}}\neq
\overline{0}}\,b_{m}= \phi(a)+\phi(b).
\end{align*}
\textit{$\phi$ is surjective}: let $f\in A^{\sigma}$, then
$f=f_{0}+f_{1}+\cdots+f_{k}$, where $k:=\deg(f)$ and $f_{i}$ is an
homogeneous polynomial of degree $i$, $0\leq i\leq k$ (if $f=0$,
then $\phi(0)=0$). Let $a=\oplus \overline{a_{m}}$, with
$a_{m}:=f_{m}$ if $0\leq m\leq k$ and $a_{m}=0$ otherwise. Then
\begin{align*}
\phi(\oplus \overline{a_{m}})=& a_{0}+a_{1}+\cdots+
a_{k}=f_{0}+f_{1}+\cdots+ f_{k}=f.
\end{align*}
\textit{$\phi$ is injective}: let $a=\oplus \overline{a_{m}}\in
Gr(A)$ such that $\phi(a)=0$, then we can represent $a$ as
$a=\overline{a_{i_{1}}}+\overline{a_{i_{2}}}+\cdots+\overline{a_{i_{s}}}$,
with $a_{i_{j}}\in F_{i_{j}}\setminus F_{i_{j}-1}$ for $1\leq
j\leq s$. Thus,
\begin{align*}
0=\phi(a)=a_{i_{1}}+a_{i_{2}}+\cdots+a_{i_{s}},
\end{align*}
and hence $a_{i_{j}}=0$ for every $1\leq j\leq s$, i.e., $a=0$.

\textit{$\phi$ is multiplicative}: By Remark \ref{identities}
(iii), we have
\begin{align}\label{621a}
&z^{\alpha}r=\sigma^{\alpha}(r)z^{\alpha}\ \text{for all}\ \alpha
\in \mathbb{N}^{n}\ \text{and}\ r\in R,\\
\label{622a} &z^{\alpha}z^{\beta}=c_{\alpha,\beta}z^{\alpha+\beta}
\ \text{for all}\ \alpha, \beta\in \mathbb{N}^{n}.
\end{align}
Let $a=\oplus_{m\geq 0}\overline{a_{m}}$ and $b=\oplus_{m\geq
0}\overline{b_{m}}$ in $Gr(A)$, then $ab=\oplus_{k\geq
0}\overline{c_{k}}$, where
\begin{align*}
\overline{c_{k}}&=\sum_{i+j=k}\overline{a_{i}}\overline{b_{j}}\\
&=\overline{a_{0}}\overline{b_{k}}+\overline{a_{1}}\overline{b_{k-1}}+\cdots+\overline{a_{k}}\overline{b_{0}}
\end{align*}
is in $F_{k}/F_{k-1}$ and $a_{i}$, $b_{j}$ are homogeneous
polynomials of degree $i$ and $j$ respectively. Since $\phi$ is
additive we only need to establish that
$\phi(\overline{a}\overline{b})=\phi(\overline{a})\phi(\overline{b})$
with $a$ and $b$ homogeneous polynomials of degree $l$ and $m$
respectively. Let
\begin{align*}
a&=c_{\alpha_{1}}x^{\alpha_{1}}+c_{\alpha_{2}}x^{\alpha_{2}}+\cdots+c_{\alpha_{r}}x^{\alpha_{r}}\\
b&=d_{\beta_{1}}x^{\beta_{1}}+d_{\beta_{2}}x^{\beta_{2}}+\cdots+d_{\beta_{s}}x^{\beta_{s}}
\end{align*}
with $|\alpha_{i}|=l$ y $|\beta_{j}|=m$ for $1\leq i\leq r$ and
$1\leq j\leq s$. Then, $ab\in F_{l+m}$ and
\begin{align*}
ab=&c_{\alpha_{1}}\sigma^{\alpha_{1}}(d_{\beta_{1}})c_{\alpha_{1},\beta_{1}}x^{\alpha_{1}+\beta_{1}}
+c_{\alpha_{1}}\sigma^{\alpha_{1}}(d_{\beta_{2}})c_{\alpha_{1},\beta_{2}}
x^{\alpha_{1}+\beta_{2}}+\cdots+\\
&c_{\alpha_{1}}\sigma^{\alpha_{1}}(d_{\beta_{s}})c_{\alpha_{1},\beta_{s}}x^{\alpha_{1}+\beta_{s}}+
\cdots+c_{\alpha_{r}}\sigma^{\alpha_{r}}(d_{\beta_{1}})c_{\alpha_{r},\beta_{1}}x^{\alpha_{r}+\beta_{1}}
+\\&
c_{\alpha_{r}}\sigma^{\alpha_{r}}(b_{\beta_{2}})c_{\alpha_{r},\beta_{2}}x^{\alpha_{r}+\beta_{2}}+\cdots+
c_{\alpha_{r}}\sigma^{\alpha_{r}}(b_{\beta_{s}})c_{\alpha_{r},\beta_{s}}x^{\alpha_{r}+\beta_{s}}+q,
\end{align*}
where $q\in A$ with $q=0$ or $\deg(q)<l+m$. From (\ref{621a}) and
(\ref{622a}) we obtain
\begin{align*}
\phi(\overline{a})\phi(\overline{b})=&(c_{\alpha_{1}}z^{\alpha_{1}}+
\cdots+c_{\alpha_{r}}z^{\alpha_{r}})(d_{\beta_{1}}z^{\beta_{1}}+\cdots+d_{\beta_{s}}z^{\beta_{s}})\\
=&c_{\alpha_{1}}z^{\alpha_{1}}d_{\beta_{1}}z^{\beta_{1}}+
\cdots+c_{\alpha_{1}}z^{\alpha_{1}}d_{\beta_{s}}z^{\beta_{s}}+\cdots+\\
&c_{\alpha_{r}}z^{\alpha_{r}}d_{\beta_{1}}z^{\beta_{1}}+ \cdots+
c_{\alpha_{r}}z^{\alpha_{r}}b_{\beta_{s}}z^{\beta_{s}}\\
=&c_{\alpha_{1}}\sigma^{\alpha_{1}}(d_{\beta_{1}})c_{\alpha_{1},\beta_{1}}z^{\alpha_{1}+\beta_{1}}
+\cdots+c_{\alpha_{1}}\sigma^{\alpha_{1}}(d_{\beta_{s}})c_{\alpha_{1},\beta_{s}}z^{\alpha_{1}+\beta_{s}}+\cdots+\\
&
c_{\alpha_{r}}\sigma^{\alpha_{r}}(d_{\beta_{1}})c_{\alpha_{r},\beta_{1}}z^{\alpha_{r}+\beta_{1}}
+\cdots+
c_{\alpha_{r}}\sigma^{\alpha_{r}}(b_{\beta_{s}})c_{\alpha_{r},\beta_{s}}z^{\alpha_{r}+\beta_{s}}\\
= & \phi(\overline{a}\,\overline{b}).
\end{align*}
Finally, we observe that $\phi(1)=1$. The last statement of the
lemma follows from Proposition \ref{associatedpbw}.
\end{proof}
\end{theorem}
The next theorem characterizes the quasi-commutative skew $PBW$
extensions.
\begin{theorem}\label{1.3.3}
Let $A$ be a quasi-commutative skew $PBW$ extension of a ring $R$.
Then,
\begin{enumerate}
\item[\rm (i)] $A$ is isomorphic to an iterated skew polynomial ring of
endomorphism type.
\item[\rm (ii)] If $A$ is bijective, then each
endomorphism is bijective.
\end{enumerate}
\begin{proof}
(i) Let $A:=\sigma(R)\langle x_1,\dots,x_n\rangle$, with
$\sigma_i$ as in Proposition \ref{sigmadefinition} (recall that
for quasi-commutative extensions each $\delta_i=0$), and let
$c_{ij}$ be the system of constants as in Definition
\ref{sigmapbwderivationtype}. Using the universal property of skew
polynomial rings (see \cite{McConnell}), we will construct the
skew polynomial ring $B:=R[z_1;\theta_1]\cdots [z_{n};\theta_n]$,
where
\begin{equation}
\left\{
\begin{aligned}
& \theta_1:=\sigma_1;\\
& \theta_j: R[z_1;\theta_1]\cdots R[z_{j-1};\theta_{j-1}]\to
R[z_1;\theta_1]\cdots R[z_{j-1};\theta_{j-1}],\\
& \theta_j(z_i):=c_{i,j}z_i, 1\leq i<j\leq n,\
\theta_j(r):=\sigma_j(r), \ \text{for} \ r\in R,
\end{aligned}
\right.
\end{equation}
and we will see that $A\cong B$.

For $n=1$, $A\cong R[z_1;\theta_1]$, with $\theta_1:=\sigma_1$. In
fact, in $A$ we have the element $y:=x_1$ that satisfies
$yg_1(r)=g_1(\theta_1(r))y$, where $g_1:R\to A$ is the ring
homomorphism defined by $g_1(r):=r$, for each $r\in R$. By the
universal property of $R[z_1;\theta_1]$ there exists a ring
homomorphism
\begin{center}
$\widetilde{g_1}:R[z_1;\theta_1]\to A$,
$\widetilde{g_1}(z_1):=y=x_1$, $\widetilde{g_1}(r):=g_1(r)=r$,
$r\in R$, i.e., $\widetilde{g_1}(p(z_1))=p(x_1)$.
\end{center}
Since $A$ is a free $R$-module, then $\widetilde{g_1}$ is an
isomorphism.

Let $n=2$. We have the ring homomorphism
\begin{center}
$f_2:R\to R[z_1;\theta_1]$, $f_2(r):=\sigma_2(r)$
\end{center}
and we set $y:=c_{1,2}z_1\in R[z_1;\theta_1]$; note that
$yf_2(r)=f_2(\theta_1(r))y$, for each $r\in R$. In fact,
$yf_2(r)=c_{1,2}z_1\sigma_2(r)=c_{1,2}\theta_1(\sigma_2(r))z_1=c_{1,2}\sigma_1\sigma_2(r)z_1$
and
$f_2(\theta_1(r))y=\sigma_2(\sigma_1(r))c_{1,2}z_1=\sigma_2\sigma_1(r)c_{1,2}z_1$
but $c_{1,2}\sigma_1\sigma_2(r)=\sigma_2\sigma_1(r)c_{1,2}$ since
$(x_2x_1)r=x_2(x_1r)$: indeed,
$(x_2x_1)r=(c_{1,2}x_1x_2)r=c_{1,2}x_1\sigma_2(r)x_2=c_{1,2}\sigma_1\sigma_2(r)x_1x_2$
and
$x_2(x_1r)=x_2\sigma_1(r)x_1=\sigma_2\sigma_1(r)x_2x_1=\sigma_2\sigma_1(r)c_{1,2}x_1x_2$.

By the universal property of $R[z_1;\theta_1]$, there exists a
ring homomorphism
\begin{center}
$\theta_2:R[z_1;\theta_1]\to R[z_1;\theta_1]$,
$\theta_2(z_1):=y=c_{1,2}z_1$, $\theta_2(r)=:f_2(r)=\sigma_2(r)$,
$r\in R$.
\end{center}
To complete the proof for the case $n=2$ we consider the ring
homomorphism $g_2':R\to A$, $g_2'(r):=r$ and let $y:=x_1\in A$;
note that $yg_2'(r)=g_2'(\theta_1(r))y$. In fact,
$yg_2'(r)=x_1r=\sigma_1(r)x_1$ and
$g_2'(\theta_1(r))y=g_2'(\sigma_1(r))x_1=\sigma_1(r)x_1$. By the
universal property of $R[z_1;\theta_1]$ there exists a ring
homomorphism
\begin{center}
$g_2:R[z_1;\theta_1]\to A$, $g_2(z_1):=y=x_1$,
$g_2(r):=g_2'(r)=r$, $r\in R$, i.e., $g_2(p(z_1))=p(x_1)$.
\end{center}
We consider in $A$ the element $y:=x_2$, note that
$yg_2(p(z_1))=g_2(\theta_2(p(z_1)))y$. In fact, we consider first
$p(z_1)=r\in R$, then $yg_2(r)=x_2r=\sigma_2(r)x_2$ and
$g_2(\theta_2(r))y=g_2(\sigma_2(r))x_2=\sigma_2(r)x_2$. Let now
$p(z_1)=z_1$, so $yg_2(z_1)=x_2x_1=c_{1,2}x_1x_2$ and
$g_2(\theta_2(z_1))y=g_2(c_{1,2}z_1)x_2=c_{1,2}x_1x_2$. Suppose by
induction that $yg_2(z_1^k)=g_2(\theta_2(z_1^k))y$, then
\begin{center}
$yg_2(z_1^{k+1})=yg_2(z_1^k)g_2(z_1)=g_2(\theta_2(z_1^k))yg_2(z_1)=$

$g_2(\theta_2(z_1^k))g_2(\theta_2(z_1))y=g_2(\theta_2(z_1^{k+1}))y$.
\end{center}
Finally,
\begin{center}
$yg_2(rz_1^k)=yg_2(r)g_2(z_1^k)=g_2(\theta_2(r))yg_2(z_1^k)=g_2(\theta_2(r))g_2(\theta_2(z_1^k))y=g_2(\theta_2(rz_1^k))$,
\end{center}
and since $g_2$ is additive, then
$yg_2(p(z_1))=g_2(\theta_2(p(z_1)))$.

By the universal property of $(R[z_1;\theta_1])[z_2;\theta_2]$,
there exists a ring homomorphism
\begin{center}
$\widetilde{g_2}:(R[z_1;\theta_1])[z_2;\theta_2]\to A$,
$\widetilde{g_2}(z_2):=y=x_2$,
$\widetilde{g_2}(p(z_1)):=g_2(p(z_1))=p(x_1)$, i.e.,
$\widetilde{g_2}(p(z_1,z_2))=p(x_1,x_2)$.
\end{center}
Since $A$ is a free $R$-module, then $\widetilde{g_2}$ is an
isomorphism.

Before the final induction step, we comment, without all details,
the case $n=3$. We define the ring homomorphism $f_3':R\to
R[z_1;\theta_1][z_2;\theta_2]$ by $f_3'(r):=\sigma_3(r)$, $r\in
R$, we set $y:=c_{1,3}z_1\in R[z_1;\theta_1][z_2;\theta_2]$ and
note that $yf_3'(r)=f_3'(\theta_1(r))y$; the universal property of
$R[z_1;\theta_1]$ induce the ring homomorphism
\begin{center}
$f_3:R[z_1;\theta_1]\to R[z_1;\theta_1][z_2;\theta_2]$,
$f_3(z_1):=y=c_{1,3}z_1$, $f_3(r):=f_3'(r)=\sigma_3(r)$, $r\in R$.
\end{center}
In order to define $\theta_3$, let $y:=c_{2,3}z_2\in
R[z_1;\theta_1][z_2;\theta_2]$, then note that
$yf_3(p(z_1))=f_3(\theta_2(p(z_1)))$. In fact, as before, we
assume first that $p(z_1)=r$, then
$yf_3(r)=c_{2,3}\sigma_2\sigma_3(r)z_2$ and
$f_3(\theta_2(r))=\sigma_3\sigma_2(r)c_{2,3}z_2$, but
$c_{2,3}\sigma_2\sigma_3(r)=\sigma_3\sigma_2(r)c_{2,3}$ since
$(x_3x_2)r=x_3(x_2r)$. For $p(z_1)=z_1$ we have
$yf_3(z_1)=c_{2,3}\sigma_2(c_{1,3})c_{1,2}z_1z_2$ and
$f_3(\theta_2(z_1))y=\sigma_3(c_{1,2})c_{1,3}\sigma_1(c_{2,3})z_1z_2$,
but
$c_{2,3}\sigma_2(c_{1,3})c_{1,2}=\sigma_3(c_{1,2})c_{1,3}\sigma_1(c_{2,3})$
since $(x_3x_2)x_1=x_3(x_2x_1)$. The rest of the proof is as
before by induction and using that $f_3$ is additive. The
universal property of $(R[z_1;\theta_1])[z_2;\theta_2]$ induces
the ring homomorphism
\begin{center}
$\theta_3:R[z_1;\theta_1][z_2;\theta_2]\to
R[z_1;\theta_1][z_2;\theta_2]$, $\theta_3(z_2):=y=c_{2,3}z_2$,
$\theta_3(p(z_1)):=f_3(p(z_1))$.
\end{center}
Now we will complete the proof for the case $n=3$. We have the
ring homomorphism $g_3'':R\to A$, $g_3''(r):=r$, we set $y:=x_1\in
A$ and note that $yg_3''(r)=g_3''(\theta_1(r))y$, so the universal
property of $R[z_1;\theta_1]$ induces the ring homomorphism
$g_3':R[z_1;\theta_1]\to A$, $g_3'(z_1):=y=x_1$,
$g_3'(r):=g_3''(r)=r$, $r\in R$, i.e., $g_3'(p(z_1))=p(x_1)$. Next
we set $y:=x_2\in A$ and observe that
$yg_3'(p(z_1))=g_3'(\theta_2(p(z_1)))$ (the proof is rutine as
before). Then, the universal property of
$(R[z_1;\theta_1])[z_2;\theta_2]$ induces the ring homomorphism
\begin{center}
$g_3:R[z_1;\theta_1][z_2;\theta_2]\to A$, $g_3(z_2):=y=x_2$,
$g_3(p(z_1)):=g_3'(p(z_1))=p(x_1)$, i.e.,
$g_3(p(z_1,z_2))=p(x_1,x_2)$.
\end{center}
Finally, we set $y:=x_3\in A$ and note that
$yg_3(p(z_1,z_2))=g_3(\theta_3(p(z_1,z_2)))y$ (the proof is as
above using double induction, for $z_1$ and for $z_2$, and using
that $g_3$ is additive). By the universal property of
$(R[z_1;\theta_1][z_2;\theta_2])[z_3;\theta_3]$ we get the ring
homomorphism
\begin{center}
$\widetilde{g_3}:(R[z_1;\theta_1][z_2;\theta_2])[z_3;\theta_3]\to
A$, $\widetilde{g_3}(z_2):=y=x_3$,
$\widetilde{g_3}(p(z_1,z_2)):=g_3(p(z_1,z_2))=p(x_1,x_2)$, i.e.,
$\widetilde{g_3}(p(z_1,z_2,z_3))=p(x_1,x_2,x_3)$.
\end{center}
Since $A$ is a free $R$-module, then $\widetilde{g_3}$ is an
isomorphism.

We will conclude the proof of the part (i) of the lemma. By
induction we can construct the ring homomorphism
\begin{center}
$\theta_n:R[z_1;\theta_1]\cdots [z_{n-1};\theta_{n-1}]\to
R[z_1;\theta_1]\cdots [z_{n-1};\theta_{n-1}]$,
$\theta_n(z_{n-1}):=c_{n-1,n}z_{n-1}$,
$\theta_n(p(z_1,\dots,z_{n-2})):=f_n(p(z_1,\dots,z_{n-2}))$,
\end{center}
where
\begin{center}
$f_n:R[z_1;\theta_1]\cdots [z_{n-1};\theta_{n-2}]\to
R[z_1;\theta_1]\cdots [z_{n-1};\theta_{n-1}]$,
$f_n(z_j):=c_{j,n}z_j$, $1\leq j\leq n-2$, $f_n(r):=\sigma_n(r)$,
\end{center}
for $r\in R$, is also defined by induction. With this we will
construct $\widetilde{g_n}$. By induction we will assume that
there exists a ring homomorphism
\begin{center}
$g_n':R[z_1;\theta_1]\cdots [z_{n-2};\theta_{n-2}]\to A$,
$g_n'(z_j):=x_j$, $1\leq j\leq n-2$, $g_n'(r):=r$, i.e.,
$g_n'(p(z_1,\dots,z_{n-2}))=p(x_1,\dots,x_{n-2})$.
\end{center}
We set $y:=x_{n-1}\in A$ and note that
\begin{center}
$yg_n'(p(z_1,\dots,z_{n-2}))=g_n'(\theta_{n-1}(p(z_1,\dots,z_{n-2})))y$.
\end{center}
In fact, $yg_n'(r)=x_{n-1}r=\sigma_{n-1}(r)x_{n-1}$ and
\[
g_n'(\theta_{n-1}(r))y=g_n'(f_{n-1}(r))x_{n-1}=g_n'(\sigma_{n-1}(r))x_{n-1}=\sigma_{n-1}(r)x_{n-1}.
\]
For each $j$ we have, $yg_n'(z_j)=x_{n-1}x_j=c_{j,n-1}x_jx_{n-1}$
and
$g_n'(\theta_{n-1}(z_j))y=g_n'(f_{n-1}(z_j))x_{n-1}=g_n'(c_{j,n-1}z_j)x_{n-1}=g_n'(c_{j,n-1})g_n'(z_j)x_{n-1}=c_{j,n-1}x_jx_{n-1}$.
By induction, and in a similar way as we saw above, we get that
$yg_n'(rz_1^{k_1}\cdots
z_{n-2}^{k_{n-2}})=g_n'(\theta_{n-1}(rz_1^{k_1}\cdots
z_{n-2}^{k_{n-2}}))y$, and since $g_n'$ is additive we conclude
that
$yg_n'(p(z_1,\dots,z_{n-2}))=g_n'(\theta_{n-1}(p(z_1,\dots,z_{n-2})))y$.
By the universal property of the iterated ring
$(R[z_1;\theta_1]\cdots
[z_{n-2};\theta_{n-2}])[z_{n-1};\theta_{n-1}]$ there exists a ring
homomorphism
\begin{center}
$g_n:(R[z_1;\theta_1]\cdots
[z_{n-2};\theta_{n-2}])[z_{n-1};\theta_{n-1}]\to A$,
$g_n(z_{n-1}):=x_{n-1}$,
$g_n(p(z_1,\dots,z_{n-2})):=g_n'(p(z_1,\dots,z_{n-2}))$, i.e.,
$g_n(p(z_1,\dots,z_{n-1}))=p(x_1,\dots,x_{n-1})$.
\end{center}
Finally, we set $y:=x_n\in A$ and we can prove as before that
\begin{center}
$yg_n(p(z_1,\dots,z_{n-1}))=g_n(\theta_n(p(z_1,\dots,z_{n-1})))y$.
\end{center}
The universal property of $(R[z_1;\theta_1]\cdots
[z_{n-1};\theta_{n-1}])[z_n;\theta_n]$ induces a ring homomorphism
$\widetilde{g_n}$ defined by
\begin{center}
$\widetilde{g_n}(z_n):=x_n$,
$\widetilde{g_n}(p(z_1,\dots,z_{n-1})):=g_n(p(z_1,\dots,z_{n-1}))=p(x_1,\dots,x_{n-1})$,
i.e., $\widetilde{g_n}(p(z_1,\dots,z_{n}))=p(x_1,\dots,x_{n})$.
\end{center}
Since $A$ is a free $R$-module, then $\widetilde{g_n}$ is an
isomorphism.

(ii) We will prove first that each $\theta_j$ is surjective: if
$r\in R$ then $r=\theta_j(\sigma_j^{-1}(r))$; let $1\leq i\leq
j-1$ and let $c:=\sigma_j^{-1}(c_{i,j}^{-1})$, then
$\theta_j(\sigma_j^{-1}(c_{i,j}^{-1})z_i)=\sigma_j(\sigma_j^{-1}(c_{i,j}^{-1}))c_{i,j}z_i=c_{i,j}^{-1}c_{i,j}z_i=z_i$.

Let $p(z_1,\dots,z_{j-1})\in R[z_1;\theta_1]\cdots
[z_{j-1};\theta_{j-1}]$ such that
$\theta_j(p(z_1,\dots,z_{j-1}))=0$. Then $p(z_1,\dots,z_{j-1})$ is a
finite summa of terms as $rz_1^{\alpha_1}\cdots
z_{j-1}^{\alpha_{j-1}}$. If we prove that each coefficient $r$ es
zero, then $p(z_1,\dots,z_{j-1})=0$ and $\theta_j$ is injective.
For this we will show that $\theta_j(rz_1^{\alpha_1}\cdots
z_{j-1}^{\alpha_{j-1}})=\sigma_j(r)uz_1^{\alpha_1}\cdots
z_{j-1}^{\alpha_{j-1}}$, where $u\in R^*$, and from this we get
$\sigma_j(r)u=0$, i.e., $\sigma_j(r)=0$, and hence, $r=0$.

Thus, $\theta_j(z_i)=c_{i,j}z_i$ and $u=c_{i,j}\in R^*$;
$\theta_j(z_i^{\alpha+1})=\theta_j(z_i^{\alpha}z_i)=\theta_j(z_i^{\alpha})\theta_j(z_i)=u'z_i^{\alpha}c_{i,j}z_i$,
where $u'$ is invertible and constructed by induction for the case
$\theta_j(z_i^{\alpha})$; hence
$\theta_j(z_i^{\alpha+1})=u'\sigma_i^{\alpha}(c_{i,j})z_i^{\alpha+1}=uz_i^{\alpha+1}$,
where $u:=u'\sigma_i^{\alpha}(c_{i,j})\in R^*$. Finally, by
induction on the number of factors, there exists $u'\in R^*$ such
that
\begin{center}
$\theta_j(rz_1^{\alpha_1}\cdots
z_{j-1}^{\alpha_{j-1}})=\sigma_j(r)\theta_j(z_1^{\alpha_1}\cdots
z_{j-2}^{\alpha_{j-2}})\theta_j(z_{j-1}^{\alpha_{j-1}})=\sigma_j(r)u'z_1^{\alpha_1}\cdots
z_{j-2}^{\alpha_{j-2}}u''z_{j-1}^{\alpha_{j-1}}=\sigma_j(r)u'\sigma_1^{\alpha_1}\cdots
\sigma_{j-2}^{\alpha_{j-2}}(u'')z_1^{\alpha_1}\cdots
z_{j-1}^{\alpha_{j-1}}$, with $u''\in R^*$.
\end{center}
Then, $u:=u'\sigma_1^{\alpha_1}\cdots
\sigma_{j-2}^{\alpha_{j-2}}(u'')\in R^*$. This complete the proof.
\end{proof}
\end{theorem}
From the above theorems we can get some interesting consequences.
\begin{corollary}[Hilbert Basis Theorem]\label{1.3.4}
Let $A$ be a bijective skew $PBW$ extension of $R$. If $R$ is a
left Noetherian ring then $A$ is also a left Noetherian ring.
\end{corollary}
\begin{proof}
According to Theorem \ref{1.3.2}, $Gr(A)$ is a quasi-commutative
skew $PBW$ extension, and by the hypothesis, $Gr(A)$ is also
bijective. By Theorem \ref{1.3.3}, $Gr(A)$ is isomorphic to an
iterated skew polynomial ring $R[z_1;\theta_1]\cdots
[z_{n};\theta_n]$ such that each $\theta_i$ is bijective, $1\leq
i\leq n$. This implies that $Gr(A)$ is a left Noetherian ring, and
hence, $A$ is left Noetherian (see \cite{McConnell}).
\end{proof}
\begin{remark}
If $A$ is a $PBW$ extension of a ring $R$, then Corollary \ref{1.3.4} extends the result in
\cite{Li} since $A$ is bijective (see also \cite{Gateva} and \cite{Zhang}). Moreover,
$A^{\sigma}=R[x_1,\dots,x_n]$ is the habitual polynomial ring since in this case $c_{i,r}=r$ and
$c_{i,j}=1$ for any $r\in R$ and every $1\leq i,j\leq n$. Thus, $Gr(A)\cong R[x_1,\dots,x_n]$.
\end{remark}
Now we can consider the regularity of bijective skew $PBW$
extensions
\begin{corollary}\label{3.2.1a}
Let $A$ be a bijective skew $PBW$ extension of a ring $R$. If $R$
is a left regular and left Noetherian ring, then $A$ is left
regular.
\end{corollary}
\begin{proof}
Theorems \ref{1.3.2} and \ref{1.3.3} say that $Gr(A)$ is
isomorphic to a iterated skew polynomial ring of automorphism type
with coefficients in $R$, then the result follows from
\cite{McConnell}, Theorem 7.7.5. and Proposition 7.7.4.
\end{proof}
Another interesting consequence of the main theorems is Serre's
Theorem for skew $PBW$ extensions.
\begin{definition}
A ring $B$ is $PSF$ if each finitely generated projective module
is stably free.
\end{definition}
\begin{corollary}[Serre's theorem]\label{3.6}
Let $A$ be a bijective skew $PBW$ extension of a ring $R$ such
that $R$ is left Noetherian, left regular and $PSF$. Then $A$ is
$PSF$.
\end{corollary}
\begin{proof}
By Theorem \ref{1.3.2}, $A$ is filtered, $A_0=R$, and $Gr(A)$ is a
quasi-commutative bijective skew $PBW$ extension of $R$; Corollary
\ref{1.3.4} says that $Gr(A)$ is left Noetherian, and Corollary
\ref{3.2.1a} implies that $Gr(A)$ is left regular. Moreover,
$Gr(A)$ is flat as right $R$-module (see Proposition
\ref{1.1.10a}), then assuming that $R$ is $PSF$ we get from
\cite{McConnell}, Theorem 12.3.2 that $A$ is $PSF$.
\end{proof}
\section{Examples of bijective skew $PBW$ extensions}\label{Examplestrad}
In this section we present and classify many remarkable examples of bijective skew $PBW$
extensions. From the results of the previous sections we can conclude that all rings presented next
are left Noetherian, left regular and $PSF$. In every example below we will highlight in bold the
ring of coefficients. However, it is important to remark that there are different ways to choose
this ring as well as the subclass to which each example belongs; $\Bbbk$ will represent a field.
\subsection{$PBW$ extensions}\label{gpbwexample}
Any \textit{PBW} extension (cf. \cite{Bell}) is a bijective skew
$PBW$ extension since in this case $\sigma_{i}=i_{R}$ for each
$1\leq i\leq n$, and $c_{i,j}=1$ for every $1\leq i,j\leq n$.
Thus, for $PBW$ extensions we have $A=i(R)\langle
x_1,\dots,x_n\rangle$. Examples of \textit{PBW} extensions are the
following (see \cite{LezamaGallego}):
\begin{enumerate}
\item[\rm (a)]The \textit{habitual polynomial ring} $A=\textbf{\emph{R}}[t_1,\dots, t_n]$.
\item[\rm (b)]Any \textit{skew polynomial ring of derivation type}
$A=\textbf{\emph{R}}[x;\sigma, \delta]$, i.e., with $\sigma=i_R$.
In general, any \textit{Ore extension of derivation type}
$\textbf{\emph{R}}[x_1;\sigma_1,\delta_1]\cdots
[x_n;\sigma_n,\delta_n]$, i.e., such that $\sigma_i=i_R$, for any
$1\leq i\leq n$. In particular, any \textit{Ore algebra of
derivation type}, i.e., when
$\boldsymbol{R:=\Bbbk[t_1,\dots,t_m]}, \ m\geq 0$.
\item[\rm (c)]The \textit{Weyl algebra} $A_n(\Bbbk):=\boldsymbol{\Bbbk[t_1,\dots, t_n]}[x_1;
\partial/\partial t_1]\cdots [x_n;
\partial/\partial t_n]$. The \textit{extended Weyl algebra}
$B_n(\Bbbk):=\boldsymbol{\Bbbk(t_1,\dots,t_n)}[x_1;
\partial/\partial t_1]\cdots [x_n;
\partial/\partial t_n]$, where $\Bbbk(t_1,\dots,t_n)$ is the field of fractions of
$\Bbbk[t_1,\dots,t_n]$, is also a $PBW$ extension. These algebras are also known as
\textit{algebras of linear partial differential operators}, see Subsection \ref{operatoralgebras}
below. The Weyl algebra $A_n(\Bbbk)$ can be generalized assuming that $R$ is an arbitrary ring,
i.e., we have the \textit{Weyl ring}
\begin{equation*}
A_n(R):=\boldsymbol{R[t_1,\dots,t_n]}[x_1;\partial/\partial
t_1]\cdots [x_n;\partial/\partial t_n].
\end{equation*}
\item[\rm (d)]Let $k$ be a
commutative ring and $\mathcal{G}$ a finite dimensional Lie
algebra over $k$ with basis $\{x_1,\dots ,x_n\}$; the
\textit{universal enveloping algebra} of $\mathcal{G}$,
$\cU(\mathcal{G})$, is a $PBW$ extension of $\boldsymbol{k}$ (see
\cite{Li}), \cite{McConnell} and \cite{Zhang}). In this case,
$x_ir-rx_i=0$ and $x_ix_j-x_jx_i=[x_i,x_j]\in
\mathcal{G}=k+kx_1+\cdots+kx_n$, for any $r\in k$ and $1\leq
i,j\leq n$.
\item[\rm (e)]Let $k$, $\mathcal{G}$, $\{x_1,\dots ,x_n\}$ and $\cU(\mathcal{G})$ be as in the
previous example; let $R$ be a $k$-algebra containing $k$. The
\textit{tensor product} $A:=R\ \otimes_k\ \cU(\mathcal{G})$ is a
$PBW$ extension of $\boldsymbol{R}$, and it is a particular case
of a more general construction, the \textit{crossed product}
$R*\cU(\mathcal{G})$ of $R$ by $\cU(\mathcal{G})$, that is also a
$PBW$ extension of $\boldsymbol{R}$ (cf. \cite{McConnell}).
\end{enumerate}
In the following subsections we consider a lot of skew $PBW$
extensions that are not \textit{PBW} extensions.
\subsection{Ore extensions of bijective type}\label{Oreextensionsofinjectivetype}
Any \textit{skew polynomial ring $\boldsymbol{R}[x;\sigma ,\delta]$ of bijective type}, i.e., with
$\sigma$ bijective, is a bijective skew $PBW$ extension (see also \cite{LezamaGallego}). In this
case we have $\boldsymbol{R}[x;\sigma,\delta]\cong\sigma(R)\langle x\rangle$. If additionally
$\delta=0$, then $R[x;\sigma]$ is quasi-commutative. More generally, let $R[x_1;\sigma_1
,\delta_1]\cdots [x_n;\sigma_n ,\delta_n]$ be an \textit{iterated skew polynomial ring of bijective
type}, i.e., the following conditions hold:
\begin{itemize}
\item for $1\leq i\leq n$, $\sigma_i$ is bijective;
\item for every $r\in R$ and $1\leq i\leq n$,
$\sigma_i(r),\delta_i(r)\in R$;
\item for $i<j$, $\sigma_j(x_i)=cx_i+d$, with $c,d\in R$ and $c$ has a
left inverse;
\item for $i<j$, $\delta_j(x_i)\in R+Rx_1+\cdots +Rx_n$,
\end{itemize}
then, $\boldsymbol{R}[x_1;\sigma_1 ,\delta_1]\cdots [x_n;\sigma_n
,\delta_n]$ is a bijective skew $PBW$ extension. Under these
conditions we have
\begin{center}
$R[x_1;\sigma_1 ,\delta_1]\cdots [x_n;\sigma_n,\delta_n]\cong \sigma(R)\langle x_1,\dots,x_n\rangle$.
\end{center}
In particular, any \textit{Ore extension $R[x_1;\sigma_1 ,\delta_1]\cdots [x_n;\sigma_n ,\delta_n]$
of bijective type}, i.e., for $1\leq i\leq n$, $\sigma_i$ is bijective, is an skew bijective $PBW$
extension. In fact, in Ore extensions for every $r\in R$ and $1\leq i\leq n$,
$\sigma_i(r),\delta_i(r)\in R$, and for $i<j$, $\sigma_j(x_i)=x_i$ and $\delta_j(x_i)=0$. An
important subclass of Ore extension of bijective type are the \textit{Ore algebras of bijective
type}, i.e., when $\boldsymbol{R=\Bbbk[t_1,\dots,t_m]}$, $m\geq 0$. Thus, we have
\begin{center}
$\Bbbk[t_1,\dots,t_m][x_1;\sigma_1 ,\delta_1]\cdots [x_n;\sigma_n
,\delta_n]\cong \sigma(\Bbbk[t_1,\dots,t_m])\langle x_1,\dots,x_n\rangle$.
\end{center}
Some concrete examples of Ore algebras of bijective type are the
following.
\begin{enumerate}
\item [\rm (a)]\label{D_qsigma} \textit{The algebra of $q$-differential operators $D_{q,h}[x,y]$}: let
$q,h\in \Bbbk, q\neq 0$; consider
$\boldsymbol{\Bbbk[y]}[x;\sigma,\delta]$, $\sigma(y):=qy$ and
$\delta(y):=h$. By definition of skew polynomial ring we have
$xy=\sigma(y)x+\delta(y)=qyx+h$, and hence $xy-qyx=h$. Therefore,
$D_{q,h}[x,y]\cong \sigma(\boldsymbol{\Bbbk[y]})\langle x\rangle$.
\item [\rm (b)]\label{shift} \textit{The algebra of shift operators $S_h$}: let
$h\in \Bbbk$. The algebra of shift operators is defined by
$S_h:=\boldsymbol{\Bbbk[t]}[x_h;\sigma_h,\delta_h]$, where
$\sigma_h(p(t)):=p(t-h)$, and $\delta_h:=0$. Thus, $S_h\cong
\sigma(\boldsymbol{\Bbbk[t]})\langle x_h\rangle$.
\item [\rm (c)]\label{mixed} The \textit{mixed algebra $D_h$}: let
$h\in \Bbbk$, the algebra $D_h$ is defined by
$D_h:=\boldsymbol{\Bbbk[t]}[x;i_{\Bbbk[t]},\frac{d}{dt}][x_h;\sigma_h,\delta_h]$, where
$\sigma_h(x):=x$. Then $D_h\cong \sigma(\boldsymbol{\Bbbk[t]})\langle x,x_h\rangle$.
\item [\rm (d)]\label{multidime} The \textit{algebra for multidimensional discrete linear systems}
is defined by
\[D:=\boldsymbol{\Bbbk[t_1,\dots,t_n]}[x_1;\sigma_1]\cdots[x_n;\sigma_n],
\]
where
\begin{center}
$\sigma_i(p(t_1,\dots,t_n)):=p(t_1,\dots,t_{i-1},t_{i}+1,t_{i+1},\dots,t_n),
\ \sigma_i(x_i)=x_i$, $1\leq i\leq n$.
\end{center}
Thus, $D$ is a quasi-commutative bijective skew $PBW$ extension of
$\boldsymbol{\Bbbk[t_1,\dots,t_n]}$.
\end{enumerate}
\subsection{Operator algebras}\label{operatoralgebras}
In this subsection we recall some important and well-known  operator algebras (cf. \cite{Chyzak1},
\cite{Wilf}).
\begin{enumerate}
\item [\rm (a)]\textit{Algebra of linear partial differential operators.}
The $n$th Weyl algebra $A_n(\Bbbk)$ over $\Bbbk$ coincides with
the $\Bbbk$-algebra of linear partial differential operators with
polynomial coefficients $\boldsymbol{\Bbbk[t_1,\dots,t_n]}$. As we
have seen, the generators of $A_n(\Bbbk)$ satisfy the following
relations
\begin{gather}
t_it_j=t_jt_i,\ \ \ \partial_i\partial_j=\partial_j\partial_i,\ \ \ \ 1\le i<j\le n,\\
\partial_jt_i=t_i\partial_j+\delta_{ij},\ \ \ \ \ \ \ \ \ \ \ \ 1\le i,j\le
n,
\end{gather}
where $\delta_{ij}$ is the Kronecker symbol. Similarly, let $\boldsymbol{\Bbbk(t_1,\dotsc,t_n)}$ be
the field of rational functions in $n$ variables. Then the $\Bbbk$-algebra of linear partial
differential operators with rational function coefficients is the algebra
$B_n(\Bbbk)=\boldsymbol{\Bbbk(t_1,\dotsc,t_n)}[\partial_1,\dotsc,\partial_n]$ where the generators
satisfy the same relations above.
\item [\rm (b)]\textit{Algebra of linear partial shift operators.} The $\Bbbk$-algebra of linear partial \textit{shift}
(\textit{recurrence}) operators with polynomial coefficients,
respectively with rational coefficients, is
$\boldsymbol{\Bbbk[t_1,\dotsc,t_n]}[E_1,\dotsc,E_m]$, respectively
$\boldsymbol{\Bbbk(t_1,\dotsc,t_n)}[E_1,\dotsc,E_m]$, $n\ge m$,
subject to the relations:
    \begin{align*}
    t_jt_i&=t_it_j,\ \ & 1\le i<j\le n,\\
    E_it_i&=(t_i+1)E_i=t_iE_i+E_i, & 1\le i\le m,\\
    E_jt_i&=t_iE_j,\ \ & i\neq j,\\
    E_jE_i&=E_iE_j,\ \ & 1\le i<j\le m.
    \end{align*}
\item [\rm (c)]\textit{Algebra of linear partial difference operators.} The $\Bbbk$-algebra of
linear partial \textit{difference opertors} with polynomial
coefficients, respectively with rational coefficients, is
$\boldsymbol{\Bbbk[t_1,\dotsc,t_n]}[\Delta_1,\dotsc,\Delta_m]$,
respectively
$\boldsymbol{\Bbbk(t_1,\dotsc,t_n)}[\Delta_1,\dotsc,\Delta_m]$,
$n\ge m$, subject to the relations:
    \begin{align*}
     t_jt_i&=t_it_j,\ \ \ & 1\le i<j\le n,\\
    \Delta_it_i&=(t_i+1)\Delta_i+1=t_i\Delta_i+\Delta_i+1,\ \ & 1\le i\le m,\\
    \Delta_jt_i&=t_i\Delta_j,\ \ & i\neq j,\\
    \Delta_j\Delta_i&=\Delta_i\Delta_j,\ \ & 1\le i<j\le m.
    \end{align*}
\item [\rm (d)]\textit{Algebra of linear partial $q$-dilation operators.} For a fixed $q\in \Bbbk-\{0\}$, the
$\Bbbk$-algebra of linear partial $q$-\textit{dilation operators}
with polynomial coefficients, respectively, with rational
coefficients, is
$\boldsymbol{\Bbbk[t_1,\dotsc,t_n]}[H_1^{(q)},\dotsc,H_m^{(q)}]$,
respectively
$\boldsymbol{\Bbbk(t_1,\dotsc,t_n)}[H_1^{(q)},\dotsc,H_m^{(q)}]$,
$n\ge m$, subject to the relations:
        \begin{align*}
        t_jt_i&=t_it_j,\ \ \ & 1\le i<j\le n,\\
        H_i^{(q)}t_i&=qt_iH_i^{(q)},\ \ & 1\le i\le m,\\
        H_j^{(q)}t_i&=t_iH_j^{(q)},\ & i\neq j,\\
        H_j^{(q)}H_i^{(q)}&=H_i^{(q)}H_j^{(q)},\ \ & 1\le i<j\le m.
        \end{align*}
\item [\rm (e)]\textit{Algebra of linear partial $q$-differential operators.} For a fixed $q\in \Bbbk-\{0\}$, the
$\Bbbk$-algebra of linear partial $q$-\textit{differential
operators} with polynomial coefficients, respectively with
rational coefficients is
$\boldsymbol{\Bbbk[t_1,\dotsc,t_n]}[D_1^{(q)},\dotsc,D_m^{(q)}]$,
respectively the ring
$\boldsymbol{\Bbbk(t_1,\dotsc,t_n)}[D_1^{(q)},\dotsc,D_m^{(q)}]$,
$n\ge m$, subject to the relations:
    \begin{align*}
    t_jt_i&=t_it_j,\ \ &1\le i<j\le n,\\
    D_i^{(q)}t_i&=qt_iD_i^{(q)}+1,\ \ &1\le i\le m,\\
    D_j^{(q)}t_i&=t_iD_j^{(q)},\ \ &i\neq j,\\
    D_j^{(q)}D_i^{(q)}&=D_i^{(q)}D_j^{(q)},\ \ &1\le i<j\le m.
    \end{align*}
Note that if $n=m$, then this operator algebra coincides with the
additive analogue $A_n(q_1,\dotsc,q_n)$ of the Weyl algebra
$A_n(q)$ (Example \ref{additivean}).
\end{enumerate}
\subsection{Diffusion algebras}\label{diffusionalgebra}
Following \cite{diffusionauthor}, a \textit{diffusion algebra}
$\cA$ is generated by $\{D_i,x_i\mid 1\le i\le n\}$ over $\Bbbk$
with relations
\begin{equation}
x_ix_j=x_jx_i, \ \ x_iD_j=D_jx_i, \ \ 1\leq i,j\leq n.
\end{equation}
\begin{equation}\label{diffusionrela}
c_{ij}D_iD_j-c_{ji}D_jD_i=x_jD_i-x_iD_j,\ \ i<j, c_{ij},c_{ji}\in K^{*}.
\end{equation}
Thus, $\cA\cong \sigma(\boldsymbol{\Bbbk[x_1,\dots,x_n]})\langle
D_1,\dotsc,D_n\rangle$.
\subsection{Quantum algebras}\label{Quantumalgebras}
\begin{enumerate}
\item [\rm (a)] \label{additivean} \textit{Additive analogue of the Weyl algebra.} The $\Bbbk$-algebra $A_n(q_1,\dots,q_n)$ is generated by
$x_1,\dots,x_n,$ $y_1,\dots,y_n$ subject to the relations:
\begin{align}
x_jx_i& = x_ix_j, & 1 \leq i,j \leq n,\\
y_jy_i&= y_iy_j, & 1 \leq i,j \leq n,\\
y_ix_j&=x_jy_i,& i\neq j,\\
y_ix_i&= q_ix_iy_i + 1,& 1\leq i\leq n,
\end{align}
where $q_i\in \Bbbk-\{0\}$. From the relations above we have
\begin{center}
$A_n(q_1, \dots, q_n)\cong \sigma(\boldsymbol{\Bbbk})\langle
x_1,\dotsc,x_n;y_1,\dots,y_n\rangle\cong
\sigma(\boldsymbol{\Bbbk[x_1,\dotsc,x_n]})\langle
y_1,\dots,y_n\rangle$.
\end{center}
\item [\rm (b)]\label{multiplicativeana} \textit{Multiplicative analogue of the Weyl algebra}.
The $\Bbbk$-algebra $\mathcal{O}_n(\lambda_{ji})$ is generated by $x_1,\dots,x_n$ subject to the relations:
\begin{center}
$x_jx_i =\lambda_{ji}x_ix_j ,\ 1\leq i<j\leq n$,
\end{center}
where $\lambda_{ji}\in \Bbbk-\{0\}$. We note that
$\mathcal{O}_n(\lambda_{ji})\cong\sigma(\Bbbk)\langle
x_1,\dotsc,x_n\rangle\cong \sigma(\boldsymbol{\Bbbk[x_1]})\langle
x_2,\dotsc,x_n\rangle$.
\item [\rm (c)]\textit{Quantum algebra $\cU'(\mathfrak{so}(3,\Bbbk))$} (cf. \cite{Havlicek} and \cite{Iorgov}).
It is the $\Bbbk$-algebra generated by $I_1,I_2,I_3$ subject to
relations
        \[
        I_2I_1-qI_1I_2=-q^{1/2}I_3,\ \ \ I_3I_1-q^{-1}I_1I_3=q^{-1/2}I_2,\ \ \
        I_3I_2-qI_2I_3=-q^{1/2}I_1,
        \]
where $q\in \Bbbk-\{0\}$. In this way, $\cU'(\mathfrak{so}(3,\Bbbk))\cong \sigma(\Bbbk)\langle I_1,
I_2, I_3 \rangle$.
\item [\rm (d)]\textit{$3$-dimensional skew polynomial algebra $\cA$.} \label{Some3dimensionall}
It is given by the relations
\begin{equation}\label{3-dimensionalconditions}
yz-\alpha zy=\lambda,\ \ \ \ zx-\beta xz=\mu,\ \ \ \ xy-\gamma
yx=\nu,
\end{equation}
such that $\lambda, \mu, \nu\in \Bbbk+\Bbbk x+\Bbbk y+\Bbbk z$, and $\alpha, \beta, \gamma \in
\Bbbk-\{0\}$. Thus, $\cA\cong \sigma(\Bbbk)\langle x,y,z\rangle$. There are fifteen $3$-dimensional
skew polynomial algebras not isomorphic. The complete list can be found in \cite{SmithBell} or
\cite{Rosenberg1}.
\item [\rm (e)]\label{dispinspbw} \textit{Dispin algebra $\cU(osp(1,2))$.} It is generated by $x,y,z$ over the commutative ring $k$ satisfying the relations
\[
yz-zy=z,\ \ \ zx+xz=y,\ \ \ xy-yx=x.
\]
Thus, $\cU(osp(1,2))\cong \sigma(\Bbbk)\langle x,y,z\rangle$.
\item [\rm (f)] \label{Worodefpbw} \textit{Woronowicz algebra $\cW_{\nu}(\mathfrak{sl}(2,\Bbbk))$.} This algebra was
introduced by Woronowicz in \cite{Woronowicz} and it is generated
by $x,y,z$ subject to the relations
\begin{gather*}
xz-\nu^4zx=(1+\nu^2)x,\ \ \ xy-\nu^2yx=\nu z,\ \ \
zy-\nu^4yz=(1+\nu^2)y,
\end{gather*}
where $\nu \in \Bbbk-\{0\}$ is not a root of unity. Then $\cW_{\nu}(\mathfrak{sl}(2,\Bbbk))\cong
\sigma (\Bbbk) \langle x,y,z\rangle$.
\item [\rm (g)]\textit{The complex algebra $V_q(\mathfrak{sl}_3(\mathbb{C}))$}. Let $q$ be a complex number
such that $q^8\neq 1$. Consider the complex algebra generated by
$e_{12},e_{13},e_{23},f_{12},f_{13},f_{23},k_1,k_2,l_1,l_2$ with
the following relations (cf. \cite{Yamane}):
\begin{align*}
    e_{13}e_{12}&=q^{-2}e_{12}e_{13}, & f_{13}f_{12}&= q^{-2}f_{12}f_{13},\\
    e_{23}e_{12}&= q^2e_{12}e_{23}-qe_{13}, & f_{23}f_{12}&=q^2f_{12}f_{23}-qf_{13},\\
    e_{23}e_{13}&=q^{-2}e_{13}e_{23}, & f_{23}f_{13}&=q^{-2}f_{13}f_{23},\\
    e_{12}f_{12}&=f_{12}e_{12}+\frac{k_1^2-l_1^2}{q^2-q^{-2}}, & e_{12}k_1&=q^{-2}k_1e_{12}, & k_1f_{12}&=q^{-2}f_{12}k_1,\\
    e_{12}f_{13}&=f_{13}e_{12}+qf_{23}k_1^2, & e_{12}k_2&=qk_2e_{12}, & k_2f_{12}&=qf_{12}k_2,\\
    e_{12}f_{23}&=f_{23}e_{12}, & e_{13}k_1&=q^{-1}k_1e_{13}, & k_1f_{13}&=q^{-1}f_{13}k_1,\\
    e_{13}f_{12}&=f_{12}e_{13}-q^{-1}l_1^2e_{23}, & e_{13}k_2&=q^{-1}k_2e_{13}, & k_2f_{13}&=q^{-1}f_{13}k_2,\\
    e_{13}f_{13}& =f_{13}e_{13}-\frac{k_1^2k_2^2-l_1^2l_2^2}{q^2-q^{-2}}, & e_{23}k_1&=qk_1e_{23}, & k_1f_{23}&=qf_{23}k_1,\\
    e_{13}f_{23}&=f_{23}e_{13}+qk_2^2e_{12}, & e_{23}k_2&=q^{-2}k_2e_{23}, & k_2f_{23}&=q^{-2}f_{23}k_2,\\
    e_{23}f_{12}&=f_{12}e_{23}, & e_{12}l_1&=q^2l_1e_{12}, & l_1f_2&=q^2f_{12}l_1,\\
    e_{23}f_{13}&=f_{13}e_{23}-q^{-1}f_{12}l_2^2, & e_{12}l_2&=q^{-1}l_2e_{12}, & l_2f_{12}&=q^{-1}f_{12}l_2,\\
    e_{23}f_{23}&=f_{23}e_{23}+\frac{k^2_2-l^2_2}{q^2-q^{-2}}, & e_{13}l_1&=ql_1e_{13}, & l_1f_{13}&=qf_{13}l_1,\\
    e_{13}l_2&=ql_2e_{13}, & l_2f_{13}&=qf_{13}l_2, & e_{23}l_1&=q^{-1}l_1e_{23},\\
    l_1f_{23}&=q^{-1}f_{23}l_1, & e_{23}l_2&=q^2l_2e_{23}, & l_2f_{23}&=q^2f_{23}l_2,\\
    l_1k_1&=k_1l_1, & l_2k_1&=k_1l_2, & k_2k_1&=k_1k_2,\\
    l_1k_2&=k_2l_1, & l_2k_2&=k_2l_2, & l_2l_1&=l_1l_2.
    \end{align*}
We can see from these relations that this algebra is a bijective
skew $PBW$ extension of the polynomial ring
$\Bbbk[l_1,l_2,k_1,k_2]$, that is,
$V_q(\mathfrak{sl}_3(\mathbb{C}))\cong
\sigma(\boldsymbol{\Bbbk[l_1,l_2,k_1,k_2]})\langle e_{12}, e_{13},
e_{23}, f_{12}, f_{13}, f_{23}\rangle$.
\item [\rm (h)]\label{algebraU} \textit{The algebra} $\textbf{U}$.
Let $\textbf{U}$ be the algebra generated over the field
$\Bbbk=\mathbb{C}$ by the set of variables $x_i,y_i,z_i$, $1\leq
i\leq n$ subject to the relations:
\begin{align*}
x_jx_i &=x_ix_j,\ y_jy_i =y_iy_j,\ z_jz_i=z_iz_j, \ 1\le i,j\le n,\\
y_jx_i&=q^{\delta_{ij}}x_iy_j,\ z_jx_i =q^{-\delta_{ij}}x_iz_j, \ 1\le i,j \le n,\\
z_jy_i &=y_iz_j, \ i\neq j,\ z_iy_i =q^{2}y_iz_i-q^2x_i^2, \ 1\le
i\le n,
\end{align*}
where $q\in \Bbbk-\{0\}$. From the relations above we see that the
algebra $\textbf{U}$ is a is a bijective skew $PBW$ extension of
$\Bbbk[x_1,\dotsc,x_n]$, that is, $
\textbf{U}\cong\sigma(\boldsymbol{\Bbbk[x_1,\dotsc,x_n]}) \langle
y_1,\dotsc,y_n;z_1,\dotsc,z_n\rangle$.
\item[\rm (i)]\label{M_qspbw} \textit{The coordinate algebra of the quantum matrix space $M_q(2)$.}
This algebra is also known as \textit{Manin algebra of $2\times 2$
quantum matrices} (cf. \cite{Li} and \cite{Man}). By definition,
$\cO_q(M_2(\Bbbk))$, also denoted $\cO(M_q(2))$, is the coordinate
algebra of the quantum matrix space $M_q(2)$, it is the
$\Bbbk$-algebra generated by the variables $x,y,u,v$
satisfying the relations
\begin{equation}\label{3com}
xu=qux,\ \ \ \ \ \ yu=q^{-1}uy,\ \ \ \ \ \ vu=uv,
\end{equation}
and
\begin{equation}\label{4com}
xv=qvx,\ \ \ \ \ \ \ vy=qyv,\ \ \ \ \ \ \ \ yx-xy=-(q-q^{-1})uv,
\end{equation}
where $q\in \Bbbk-\{0\}$. Thus, $\cO(M_q(2))\cong \sigma(\boldsymbol{\Bbbk[u]})\langle
x,y,v\rangle$. Due to the last relation in (\ref{4com}), we remark that it is not possible to
consider $\cO(M_q(2))$ as an skew $PBW$ extension of $\Bbbk$. This algebra can be generalized to
$n$ variables, $\cO_q(M_n(\Bbbk))$, and coincides with the \textit{coordinate algebra of the
quantum group $SL_q(2)$}, see \cite{Gomez-Torrecillas2} for more details.
\item [\rm (j)]\label{QHeispbw} \textit{$q$-Heisenberg algebra}. The $\Bbbk$-algebra \textbf{H}$_n(q)$ is generated by the set of variables $x_1,\dots,x_n,$ $y_1,\dots,y_n,$ $z_1,\dots,z_n$ subject to the
relations:
\begin{align}
x_jx_i &= x_ix_j, &z_jz_i&=z_iz_j, & y_jy_i &= y_iy_j, \ \ 1 \leq
i,j
\leq n,\label{4.10}\\
z_jy_i&=y_iz_j, &z_jx_i&=x_iz_j, &y_jx_i&=x_iy_j, \ \ i\neq j,\label{4.11}\\
z_iy_i &= qy_iz_i, &z_ix_i&=q^{-1}x_iz_i+y_i, & y_ix_i &= qx_iy_i,
\ 1\leq i\leq n,\label{4.12}
\end{align}
with $q\in \Bbbk-\{0\}$. Note that
\[
\textbf{H}_n(q)\cong \sigma(\Bbbk)\langle
x_1,\dots,x_n;y_1,\dots,y_n;z_1,\dots,z_n\rangle\cong\sigma(\boldsymbol{\Bbbk[y_1,\dotsc,y_n]})\langle
x_1,\dotsc,x_n;z_1,\dotsc,z_n\rangle.
\]
\item [\rm (k)] \label{qeasl2k} \textit{Quantum enveloping algebra of $\mathfrak{sl}(2,\Bbbk)$.} $\cU_q(\mathfrak{sl}(2,\Bbbk))$ is defined
as the algebra generated by $x,y,z,z^{-1}$ with relations
\begin{gather}
zz^{-1}= z^{-1}z=1, \label{Kassel1.10}\\
xz= q^{-2}zx,\ \ \ yz= q^{2}zy, \label{Kassel1.11}\\
xy-yx= \frac{z-z^{-1}}{q-q^{-1}}, \label{Kassel1.12}
\end{gather}
where $q\neq 1,-1$. The above relations show that
$\cU_q(\mathfrak{sl}(2,\Bbbk))=\sigma(\boldsymbol{\Bbbk[z,z^{-1}]})\langle x,y\rangle$. This
example can be extended to the quantized enveloping algebra $U_q(\mathcal{L})$ of a complex
semisimple finite-dimensional Lie algebra $\mathcal{L}$ with Cartan matrix $[a_{ij}]$, such that
$a_{ij}=0$ for $i\neq j $ (see \cite{Brown}, and also \cite{Kassel}).
\item [\rm (l)]\textit{Hayashi algebra $W_q(J)$.}\label{Hayhypspbw}
T. Hayashi in \cite{Ha} defined the quantized $\Bbbk$-algebra $W_q(J)$ generated by $x_i, y_i, z_i,
1\leq i\leq n$, subject to the relations (\ref{4.10})-(\ref{4.12}) replacing
$z_ix_i=q^{-1}x_iz_i+y_i$ by
\begin{equation}\label{hayres}
(z_ix_i-qx_iz_i)y_i=1=y_i(z_ix_i-qx_iz_i),\ \ i=1,\dotsc,n,
\end{equation}
with $q\in \Bbbk-\{0\}$. Note that $W_q(J)$ is a $\sigma-PBW$ extension of $\Bbbk[y_1^{\pm
1},\dotsc,y_n^{\pm 1}]$. In fact,
\begin{align}
x_iy_j^{-1}=y_j^{-1}x_i,\ z_iy_j^{-1}=y_j^{-1}z_i,\
y_jy_j^{-1}=y^{-1}_jy_j=1,\ z_ix_i=qx_iz_i+y_i^{-1}, \ 1\le i,j\le
n\label{ggfz}.
\end{align}
Hence, $W_q(J)\cong\sigma(\boldsymbol{\Bbbk[y_1^{\pm
1},\dotsc,y_n^{\pm 1}]})\langle
x_1,\dotsc,x_n;z_1,\dotsc,z_n\rangle$.
\item [\rm (m)]\label{e.1.3spbw} \textit{The algebra of differential operators} $D_{\textbf{q}}(S_{\textbf{q}})$ \textit{on a quantum space} $S_{\textbf{q}}$.
Let $k$ be a commutative ring and let $\textbf{q}=[q_{ij}]$ be a matrix with entries in $k^*$ such
that $q_{ii}=1=q_{ij}q_{ji}$ for all $1\leq i,j \leq n$. The $k$-algebra $S_{\textbf{q}}$ is
generated by $x_i, 1\leq i\leq n$, subject to the relations
\begin{equation}\label{nre}
x_ix_j=q_{ij}x_jx_i.
\end{equation}
The algebra $S_{\textbf{q}}$ is regarded as the \textit{algebra of
functions on a quantum space} (see also Section \ref{sec5}). The
algebra $D_{\textbf{q}}(S_{\textbf{q}})$ of
$\textbf{q}$-differential operators on $S_{\textbf{q}}$ is defined
by
\begin{equation}\label{wfwrrg}
\partial_ix_j-q_{ij}x_j\partial_i=\delta_{ij}\ \ \textrm{for\
all}\ i,j.
\end{equation}
The relations between $\partial_i$, are given by
\begin{equation}\label{nwq}
\partial_i\partial_j=q_{ij}\partial_j\partial_i,\ \ \textrm{for\
all}\ i,j.
\end{equation}
Therefore, $D_{\textbf{q}}(S_{\textbf{q}})\cong
\sigma(\boldsymbol{\sigma(k)\langle x_1,\dotsc,x_n\rangle)}\langle
\partial_1,\dotsc,\partial_n\rangle$. More exactly,
$D_{\textbf{q}}(S_{\textbf{q}})$ is a bijective skew $PBW$ extension of a quasi-commutative
bijective skew $PBW$ extension of a commutative ring $k$.
\item [\rm (n)]\textit{Witten's deformation of} $\cU(\mathfrak{sl}(2,\Bbbk)$. E. Witten introduced and
studied a 7-parameter deformation of the universal enveloping
algebra $\cU(\mathfrak{sl}(2,\Bbbk))$ depending on a 7-tuple of
parameters $\underline{\xi}=(\xi_1,\dotsc,\xi_7)$ and subject to
relations
    \[
    xz-\xi_1zx=\xi_2x,\ \ \ zy-\xi_3yz=\xi_4y,\ \ \ \ yx-\xi_5xy=\xi_6z^2+\xi_7z.
    \]
    The resulting algebra is denoted by $W(\underline{\xi})$. In \cite{Levandovskyy} it is assumed that $\xi_1\xi_3\xi_5\neq 0$.
    Note that that $W(\underline{\xi})\cong \sigma(\boldsymbol{\sigma(\Bbbk[x])\langle z\rangle})\langle y\rangle$.
\item [\rm (o)] \label{QWMalt} \textit{Quantum Weyl algebra of Maltsiniotis} $A_n^{\textbf{q},\lambda}$. Let $k$ be a commutative ring and
$\textbf{q}=[q_{ij}]$ a matrix over $k$ such that $q_{ij}q_{ji}=1$ and $q_{ii}=1$ for all $1\le
i,j\le n$. Fix an element $\lambda:=(\lambda_1,\dotsc, \lambda_n)$ of $(k^*)^n$. By definition,
this algebra is generated by the variables $x_i,y_j,\ 1\le i,j\le n$ subject to the relations
\begin{enumerate}
\item [\rm (a)]For any $1\le i < j\le n$,
\begin{align*}
x_ix_j&=\lambda_iq_{ij}x_jx_i, & y_iy_j & =q_{ij}y_jy_i,\\
x_iy_j&=q_{ji}y_jx_i, & y_ix_j & =\lambda_i^{-1}q_{ji}x_jy_i.
\end{align*}
\item [\rm (b)]For any $1\le i\le n$,
\begin{equation*}
x_iy_i-\lambda_iy_ix_i=1+\sum_{1\le j< i}(\lambda_j-1)y_jx_j.
\end{equation*}
\end{enumerate}
From relations above we have that $A_n^{\textbf{q},\lambda}$ is isomorphic to a bijective skew
$PBW$ extension,
\begin{center}
$A_n^{\textbf{q},\lambda}\cong \sigma(\boldsymbol{\sigma(\cdots \sigma(\sigma(k)\langle
x_1,y_1\rangle)\langle x_2,y_2\rangle)\cdots)\langle x_{n-1},y_{n-1}\rangle})\langle
x_n,y_n\rangle$.
\end{center}
\item [\rm (p)] \textit{Quantum Weyl algebra $A_n(q,p_{i,j})$.} The ring $A_n(q,p_{i,j})$ arising in
\cite{Giaquinto} as $A_n(R)$ for the \textquotedblleft standard"
multiparameter Hecke symmetry. This ring can be viewed as a
quantization of the usual Weyl algebra $A_n(\Bbbk)$. By
definition, $A_n(q,p_{i,j})$ is the ring generated over the field
$\Bbbk$ by the variables $x_i,\partial_j$ with $i,j=1,\dotsc,n$
and subject to relations
    \begin{align*}
    & x_ix_j=p_{ij}qx_jx_i, & {\rm for\ all}\ i<j,\\
    & \partial_i\partial_j=p_{ij}q^{-1}\partial_j\partial_i, & {\rm for\ all}\ i<j,\\
    & \partial_ix_j=p_{ij}^{-1}qx_j\partial_i, & {\rm for\ all}\ i\neq j,\\
    & \partial_i x_i=1+q^2x_i\partial_i+(q^2-1)\sum_{i<j}x_j\partial_j, & {\rm for\ all}\ i.
    \end{align*}
When $q=1$ and each $p_{ij}=1$, these relations give the usual Weyl algebra $A_n(\Bbbk)$. From
relations above we have that $A_n(q,p_{i,j})$ is a bijective skew $PBW$ extension,
\begin{center}
$A_n(q,p_{i,j})\cong \sigma(\boldsymbol{\sigma(\cdots \sigma(\sigma(\Bbbk)\langle
x_{n},\partial_{n}\rangle)\langle x_{n-1},\partial_{n-1}\rangle)\cdots)\langle
x_{2},\partial_{2}\rangle})\langle x_1,\partial_1\rangle$.
\end{center}
\item [\rm (q)] \textit{Multiparameter quantized Weyl algebra} $A_n^{Q,\Gamma}(\Bbbk)$. Let $Q:=[q_1,\dotsc,q_n]$ be a vector in $\Bbbk^{n}$
with no zero components, and let $\Gamma=[\gamma_{ij}]$ be a multiplicatively antisymmetric
$n\times n$ matrix over $\Bbbk$. The multiparameter Weyl algebra $A_n^{Q,\Gamma}(\Bbbk)$ is defined
to be the algebra generated by $\Bbbk$ and the indeterminates $y_1,\dotsc,y_n,x_1,\dotsc,x_n$
subject to the relations:
    \begin{align*}
    & y_iy_j=\gamma_{ij}y_jy_i, & 1\le i,j\le n,\\
    & x_ix_j=q_i\gamma_{ij}x_jx_i, & 1\le i<j\le n,\\
    & x_iy_j=\gamma_{ji}y_jx_i, & 1\le i<j\le n,\\
    & x_iy_j=q_j\gamma_{ji}y_jx_i, & 1\le j<i\le n,,\\
    & x_jy_j=q_jy_jx_j +1 +\sum_{l<j}(q_l-1)y_lx_l, & 1\le j\le n.
    \end{align*}
From relations above we have that $A_n^{Q,\Gamma}(\Bbbk)$ is isomorphic to a bijective skew $PBW$
extension,
\begin{center}
$A_n^{Q,\Gamma}(\Bbbk)\cong \sigma(\boldsymbol{\sigma(\cdots \sigma(\sigma(k)\langle
x_1,y_1\rangle)\langle x_2,y_2\rangle)\cdots)\langle x_{n-1},y_{n-1}\rangle})\langle
x_n,y_n\rangle$.
\end{center}
\item [\rm (r)] \textit{Quantum symplectic space} $\cO_q(\mathfrak{sp}(\Bbbk^{2n}))$. For every nonzero element $q$ in $\Bbbk$,
one defines this quantum algebra $\cO_q(\mathfrak{sp}(\Bbbk^{2n}))$ to be the algebra generated by $\Bbbk$ and the variables $y_1,\dotsc,y_n,x_1,\dotsc,x_n$ subject to the relations
\begin{align*}
y_jx_i& =q^{-1}x_iy_j, \ \ \ \ \ \ \ \ \ \ y_jy_i=qy_iy_j, & 1\le i<j\le n,\\
x_jx_i& =q^{-1}x_ix_j, \ \ \ \ \ \ \ \ \ x_jy_i=qy_ix_j, & 1\le i<j\le n,\\
x_iy_i-q^2y_ix_i& =(q^2-1)\sum_{l=1}^{i-1}q^{i-l}y_lx_l, & 1\le i\le n.
\end{align*}
From relations above we have that $\cO_q(\mathfrak{sp}(\Bbbk^{2n}))$ is isomorphic to a bijective
skew $PBW$ extension,
\begin{center}
$\cO_q(\mathfrak{sp}(\Bbbk^{2n}))\cong \sigma(\boldsymbol{\sigma(\cdots \sigma(\sigma(k)\langle
x_1,y_1\rangle)\langle x_2,y_2\rangle)\cdots)\langle x_{n-1},y_{n-1}\rangle})\langle
x_n,y_n\rangle$.
\end{center}
\end{enumerate}
\subsection{Quadratic algebras in 3 variables}\label{quadraticalgebras3varia}
Quadratic algebras in 3 variables are considered as a class of
$G$-algebras in 3 variables which relations are homogeneous of
degree 2 (cf. \cite{Levandovskyy} for more details). More exactly,
a quadratic algebra in 3 variables $\cA$ is a $\Bbbk$-algebra
generated by $x,y,z$ subject to the relations
\begin{align*}
yx & = xy+a_1z+a_2y^2+a_3yz+\xi_1z^2,\\
zx & = xz+\xi_2y^2+a_5yz+a_6z^2,\\
zy & =yz+a_4z^2.
\end{align*}
Note that these algebras are examples of bijective
$\sigma$-\textit{PBW} extensions. In fact, consider the bijective
skew \textit{PBW} extension $\sigma(\Bbbk[z])\langle y\rangle$ of
$\Bbbk[z]$ defined by the relation $yz-zy=-a_4z^2$. Then
$\cA\cong \sigma(\boldsymbol{\sigma(\Bbbk[z])\langle
y\rangle)}\langle x\rangle$.
\begin{remark}
(i) From Corollaries \ref{1.3.4}, \ref{3.2.1a}, \ref{3.6}, and assuming that $\boldsymbol{R}$ and
$\boldsymbol{k}$ are left Noetherian, left regular and $PSF$ rings, we can conclude that all rings
and algebras presented above are left Noetherian, left regular and $PSF$ rings. Li (\cite{Li})
shows that linear solvable polynomial algebras are left Noetherian, left regular and \textit{PSF}
rings (all these algebras considered over a field $\Bbbk$). We want to remark that the results
obtained in the present paper generalized the results of Li because we deal with extensions of
rings more general than fields.

(ii) Viktor Levandovskyy has defined in \cite{Levandovskyy} the $G$-algebras; let $\Bbbk$ be a
field, a $\Bbbk$-algebra $A$ is called a \textit{$G$-algebra} if $\Bbbk \subset Z(A)$ (center of
$A$) and $A$ is generated by a finite set $\{x_1, \ldots, x_n\}$ of elements that satisfy the
following conditions: (a) the collection of standard monomials of $A$ is a $\Bbbk$-basis of $A$.
(b) $x_jx_i= c_{ij}x_ix_j+d_{ij}$, for $1 \leq i < j \leq n$, with $c_{ij} \in \Bbbk-\{0\}$ and
$d_{ij}\in A$. (c) There exists a total order $<_A$ on ${\rm Mon}(A)$ such that for $i<j$,
$lm(d_{ij}) <_A x_ix_j$. According to this definition, $G$-algebras appear like more general than
skew $PBW$ extensions since $d_{ij}$ is not necessarily linear, however, in $G$-algebras the
coefficients of polynomials are in a field and they commute with the variables $x_1,\dots,x_n$.
However, note that the class of $G$-algebras does not include the class of skew $PBW$ extensions
over fields. For example, consider the $\Bbbk$-algebra $\cA$ generated by $x,y,z$ subject to the
relations
\[
yx-q_2xy=x,\ \ \ \ \ \ zx-q_1xz=z,\ \ \ \ \ \ \ \ zy=yz, \ \ \ \ q_1,q_2\in \Bbbk.
\]
Then $\cA$ is not a $G$-algebra in the sense of \cite{Levandovskyy}. Note that if $q_1,q_2\neq 0$,
then $\cA\cong \sigma(\Bbbk)\langle x,y,z\rangle$.

Another interesting and non trivial example is related to Witten's algebras: In \cite{Levandovskyy}
it is shown that the only possible $G$-algebra with all three $q$-commutators being nonzero is of
the form
    \[
    xz-qzx=\xi_2x,\ \ \ zy-qyz=\xi_2y,\ \ \ \ yx-\xi_5xy=\xi_6z^2+\xi_7z,
    \]
    that is, considering $q=\xi_1=\xi_3$ and $\xi_2=\xi_4$ with $q,\xi_2,\xi_5\in \Bbbk-\{0\}$ and $\xi_6,\xi_7\in \Bbbk$.

(iii) A similar remark can be done with respect to $PBW$ rings and algebras defined by Bueso,
Gómez-Torrecillas and Verschoren in \cite{Gomez-Torrecillas2}. For any $0\neq q\in \Bbbk$, let $R$
be an algebra generated by the variables $a,b,c,d$ subject to the relations
\begin{align*}
ba&=qab,\ \ \ \ db=qbd,\ \ \ \ ca=qac,\ \ \ \ dc=qcd\\
bc&=\mu cb,\ \ \ \ ad-da=(q^{-1}-q)bc.
\end{align*}
for some $\mu\in \Bbbk$. Then $R$ is not a $PBW$ ring (and hence is not a $PBW$ algebra) unless
$\mu=1$ (see \cite{Gomez-Torrecillas2}). Note that this algebra is a bijective skew $PBW$ extension
of $\Bbbk[b]$, that is, isomorphic to $\sigma(\Bbbk[b])\langle a,c,d\rangle$.
\end{remark}

\subsection{Skew quantum polynomials over rings}\label{sec5}
In this subsection we define the \textit{skew quantum polynomials over rings} and we will show that
the \textit{algebra of quantum polynomials} defined in \cite{Artamonov} and \cite{Artamonov2} is a
particular case of this kind of non-commutative rings. We will see also that the ring of skew
quantum polynomials is a localization of a quasi-commutative bijective skew $PBW$ extension. From
this, and using the results of the previous sections, we can conclude that these rings are left
Noetherian, left regular and $PSF$.

\begin{example}\label{1.4.2}
Let $R$ be a ring with a fixed matrix of parameters
$\textbf{q}:=[q_{ij}]\in M_n(R)$, $n\geq 2$, such that
$q_{ii}=1=q_{ij}q_{ji}=q_{ji}q_{ij}$ for every $1\leq i,j\leq n$,
and suppose also that it is given a system
$\sigma_1,\dots,\sigma_n$ of automorphisms of $R$. The ring of
\textit{skew quantum polynomials over $R$}, denoted by
$R_{\textbf{q},\sigma}[x_1^{\pm 1 },\dots,x_r^{\pm 1},
x_{r+1},\dots,x_n]$, is defined as follows:
\begin{enumerate}
\item[\rm{(i)}]$R\subseteq R_{\textbf{q},\sigma}[x_1^{\pm 1},\dots,x_r^{\pm 1},
x_{r+1},\dots,x_n]$;
\item[\rm{(ii)}]$R_{\textbf{q},\sigma}[x_1^{\pm 1
},\dots,x_r^{\pm 1}, x_{r+1},\dots,x_n]$ is a free left $R$-module
with basis
\begin{equation}\label{equ1.4.2}
\{x_1^{\alpha_1}\cdots x_n^{\alpha_n}|\alpha_i\in \mathbb{Z} \
\text{for}\ 1\leq i\leq r \ \text{and} \ \alpha_i\in \mathbb{N}\
\text{for}\ r+1\leq i\leq n\};
\end{equation}
\item[\rm{(iii)}] the variables $x_1,\dots,x_n$ satisfy the defining relations
\begin{center}
$x_ix_i^{-1}=1=x_i^{-1}x_i$, $1\leq i\leq r$,

$x_jx_i=q_{ij}x_ix_j$, $x_ir=\sigma_i(r)x_i$, $r\in R$, $1\leq
i,j\leq n$.
\end{center}
\end{enumerate}
When all automorphisms are trivial, we write
$R_{\textbf{q}}[x_1^{\pm 1 },\dots,x_r^{\pm 1},
x_{r+1},\dots,x_n]$, and this ring is called the ring of
\textit{quantum polynomials over $R$}. If $R=\Bbbk$ is a field,
then $\Bbbk_{\textbf{q},\sigma}[x_1^{\pm 1 },\dots,x_r^{\pm 1},
x_{r+1},\dots,x_n]$ is the \textit{algebra of skew quantum
polynomials}. For trivial automorphisms we get the \textit{algebra
of quantum polynomials} simply denoted by $\mathcal{O}_\textbf{q}$
(see \cite{Artamonov}).

$R_{\textbf{q},\sigma}[x_1^{\pm 1},\dots,x_r^{\pm 1}, x_{r+1},\dots,x_n]$ can be viewed as a
localization of an skew $PBW$ extension. In fact, we have the quasi-commutative bijective skew
$PBW$ extension
\begin{center}
$A:=\sigma(R)\langle x_1,\dots, x_n\rangle$,\ \ with\ \
$x_ir=\sigma_i(r)x_i$\ \ and\ \ $x_jx_i=q_{ij}x_ix_j$, $1\leq
i,j\leq n$.
\end{center}
If we set
\begin{center}
$S:=\{rx^{\alpha}\mid r\in R^*, x^{\alpha}\in {\rm Mon}\{x_1,\dots,x_r\}\}$,
\end{center}
then $S$ is a multiplicative subset of $A$ and
\begin{center}
$S^{-1}A\cong R_{\textbf{q},\sigma}[x_1^{\pm 1},\dots,x_r^{\pm 1},
x_{r+1},\dots,x_n]$.
\end{center}
In fact, if $f\in A$ and $rx^{\alpha}\in S$ are such that $frx^{\alpha}=0$, then
$0=frx^{\alpha}=fx^{\alpha}[(\sigma^{\alpha})^{-1}(r)]$, so $0=fx^{\alpha}$ since
$(\sigma^{\alpha})^{-1}(r)\in R^*$, and hence, $f=0$. From this we get that $rx^{\alpha}f=0$. $S$
satisfies the left (right) Ore condition: if $ f=c_1x^{\beta_1}+\cdots+c_tx^{\beta_t}$, then
$grx^{\alpha}=x^{\alpha}f$, where $g:=d_1x^{\beta_1}+\cdots+d_tx^{\beta_t}$ with
$d_i:=\sigma^{\alpha}(c_i)c_{\alpha,\beta_i}c_{\beta_i,\alpha}^{-1}\sigma^{\beta_i}(r^{-1})$, and
$c_{\alpha,\beta_i}$, $c_{\beta_i,\alpha}$ are the elements of $R$ that we obtain when we applying
Theorem \ref{coefficientes} to $A$ (for the right Ore condition $g$ is defined in a similar way).
This means that $S^{-1}A$ exists ($AS^{-1}$ also exists, and hence, $S^{-1}A\cong AS^{-1}$).

Finally, note that the function
\begin{center}
$h':A\to R_{\textbf{q},\sigma}[x_1^{\pm 1},\dots,x_r^{\pm 1},
x_{r+1},\dots,x_n]$, $h'(f):=f$
\end{center}
is a ring homomorphism and satisfies $h'(S)\subseteq
R_{\textbf{q},\sigma}[x_1^{\pm 1},\dots,x_r^{\pm 1},
x_{r+1},\dots,x_n]^*$ (in fact,
$[rx^{\alpha}]^{-1}=(\sigma^\alpha)^{-1}(r^{-1})(x^{\alpha})^{-1}$),
so $h'$ induces the ring homomorphism
\begin{center}
$h:S^{-1}A\to R_{\textbf{q},\sigma}[x_1^{\pm 1},\dots,x_r^{\pm 1},
x_{r+1},\dots,x_n]$,
$h(\frac{f}{rx^{\alpha}}):=h'(rx^{\alpha})^{-1}h'(f)=(rx^{\alpha})^{-1}f.$
\end{center}
It is clear that $h$ is injective; moreover, $h$ is surjective
since $x_i=h(\frac{x_i}{1})$, $1\leq i\leq n$,
$x_j^{-1}=h(\frac{1}{x_j})$, $1\leq j\leq r$, $r=h(r)$, $r\in R$.
\end{example}
When $r=0$, $R_{\textbf{q},\sigma}[x_1^{\pm 1},\dots,x_r^{\pm 1},
x_{r+1},\dots,x_n]=R_{\textbf{q},\sigma}[x_1,\dots,x_n]$ is the
\textit{$n$-mul\-ti\-pa\-ra\-me\-tric skew quantum space over $R$}, and when $r=n$, it coincides
with $R_{\textbf{q},\sigma}[x_1^{\pm 1},\dots,x_n^{\pm 1}]$, i.e., with the
\textit{$n$-multiparametric skew quantum torus over $R$}. In this case, if $n=1$,
$R_{\textbf{q},\sigma}[x_1^{\pm 1},\dots,x_n^{\pm 1}]=R[x^{\pm 1};\sigma]$, i.e., this ring
coincides with the \textit{skew Laurent polynomial ring over $R$}.
\begin{remark}\label{3.3a}
If $R$ is a left Noetherian, left regular and $PSF$, then $R_{\textbf{q},\sigma}[x_1^{\pm
1},\dots,x_r^{\pm 1}, x_{r+1},\dots,x_n]$ is left Noetherian, left regular and $PSF$: this is
consequence also of Corollaries \ref{1.3.4}, \ref{3.2.1a}, \ref{3.6}, Example \ref{1.4.2} and also
from the behavior of the noetherianity, regularity and $PSF$ condition under the ring of fractions
(see \cite{Bell2} and \cite{McConnell}). Compare this result with \cite{Artamonov} and
\cite{Artamonov2}.
\end{remark}
\section{Dimensions}
In this section we estimate the global, Krull and Goldie dimensions of bijective skew \textit{PBW}
extensions. A preliminary elementary property of skew $PBW$ extensions is needed.
\begin{proposition}\label{1.1.10}
Let $A$ be an skew $PBW$ extension of a ring $R$. If $R$ is a domain, then $A$ is a domain.
\end{proposition}
\begin{proof}
This follows from Theorem \ref{1.3.2}, Theorem \ref{1.3.3}, and Theorem 1.2.9 and Proposition 1.6.6
in \cite{McConnell}. However, we can give also a direct proof. Let $f=cx^{\alpha}+p$,
$g=dx^{\beta}+ q$ be non zero elements of $A$, with $cx^{\alpha}=lt(f)$, $dx^{\beta}=lt(g)$, so
$c,d\neq 0$, $x^{\alpha}\succ lm(p)$ and $x^{\beta}\succ lm(q)$. We get
\begin{align*}
fg&=(cx^{\alpha}+p)(dx^{\beta}+ q)=cx^{\alpha}dx^{\beta}+cx^{\alpha}q+pdx^{\beta}+pq
  =c(d_{\alpha}x^{\alpha}+p_{\alpha,d})x^{\beta}+cx^{\alpha}q+pdx^{\beta}+pq,
\end{align*}
with $0\neq d_{\alpha}=\sigma^{\alpha}(d)\in R$, $p_{\alpha,d}\in A$, $p_{\alpha,d}=0$ or
$\deg(p_{\alpha,d})<|\alpha|$. Hence,
\begin{align*}
fg&=cd_{\alpha}x^{\alpha}x^{\beta}+cp_{\alpha,d}x^{\beta}+cx^{\alpha}q+pdx^{\beta}+pq
  =cd_{\alpha}(c_{\alpha,\beta}x^{\alpha+\beta}+p_{\alpha,\beta})+cp_{\alpha,d}x^{\beta}+cx^{\alpha}q+pdx^{\beta}+pq\\
  &=cd_{\alpha}c_{\alpha,\beta}x^{\alpha+\beta}+cd_{\alpha}p_{\alpha,\beta}+cp_{\alpha,d}x^{\beta}
  +cx^{\alpha}q+pdx^{\beta}+pq,
\end{align*}
where $0\neq c_{\alpha,\beta}\in R$, $p_{\alpha,\beta}\in A$, $p_{\alpha,\beta}=0$ or
$\deg(p_{\alpha,\beta})<|\alpha+\beta|$. Moreover $cd_{\alpha}c_{\alpha,\beta}\neq 0$ and
$h:=cd_{\alpha}p_{\alpha,\beta}+cp_{\alpha,d}x^{\beta} +cx^{\alpha}q+pdx^{\beta}+pq\in A$ is such
that $h=0$ or $x^{\alpha+\beta}\succ lm(h)$. This means that $fg\neq 0$.
\end{proof}
\begin{theorem}\label{3.1.1a}
Let $A$ be a bijective skew \textit{PBW} extension of a ring $R$. Then
\begin{enumerate}
\item[\rm{(i)}]
\begin{gather}
{\rm lgld}(R)\le {\rm lgld}(A)\le {\rm lgld}(R)+n,\ \ {\rm if}\ \ {\rm lgld}(R)\le \infty.
\end{gather}
If $A$ is quasi-commutative, then
\[
{\rm lgld}(A)={\rm lgld}(R)+n.\] In particular, if $R$ is semisimple, then ${\rm lgld}(A)=n$.
\item[\rm{(ii)}]If $R$ is left
Noetherian, then
\begin{gather} {\rm lKdim}(R)\le {\rm lKdim}(A)\le
{\rm lKdim}(R)+n.
\end{gather}
If $A$ is quasi-commutative, then
\[
{\rm lKdim}(A)={\rm lKdim}(R)+n.
\]
In particular, if $R=\Bbbk$ is a field, then ${\rm lKdim}(A)=n$.
\item[\rm{(iii)}]If $R$ is a left Noetherian domain, then the Goldie dimension
of $A$ is {\rm 1}, that is, ${\rm ludim}\ A=1$.
\end{enumerate}
\begin{proof}
(i) Since $A$ is a filtered ring, then by \cite{McConnell} Corollary 7.6.28
\[
{\rm lgld}(A)\le {\rm lgld}(Gr(A)).
\]
According to Theorems \ref{1.3.2} and \ref{1.3.3}, $Gr(A)$ is isomorphic to an iterated skew
polynomial ring of automorphism type. From \cite{McConnell} Theorem 7.5.3 we get the right
inequalities. By Proposition \ref{1.1.10a}, $A_R$ is free, and hence projective and faithfully
flat. From \cite{McConnell} Theorem 7.2.6 we get the left inequalities. If $A$ is quasicommutative,
from Theorem \ref{1.3.3} we get the equalities. Finally, if $R$ is semisimple, ${\rm
lgld}(R)=0={\rm lgld}(R)$.

(ii) From Corollary \ref{1.3.4} we know that $A$ is a left Noetherian ring. Now, using the fact
that $A_R$ is free, and hence faithfully flat, from \cite{McConnell} Corollary 6.5.3 and Lemma
6.5.6 we get
\begin{gather}
{\rm lKdim}(R)\le {\rm lKdim}(A)\le {\rm lKdim}(Gr(A)).
\end{gather}
Again, by Theorems \ref{1.3.2} and \ref{1.3.3}, we have the first assertion. Now, if $A$ is
quasi-commutative we apply directly Theorem \ref{1.3.3} and \cite{McConnell} Proposition 6.5.4.

(iii) By Corollary \ref{1.3.4} and Proposition \ref{1.1.10}, $A$ is a left Noetherian domain, but
it is well known that every left Noetherian domain is a left Ore domain (\cite{McConnell}). The
assertion follows from the fact that for a domain $S$, ${\rm ludim}\ S=1$ if and only if $S$ is a
left Ore domain (cf. \cite{McConnell}, Example 2.2.11).
\end{proof}
\end{theorem}
\begin{theorem}
Let $R$ be a ring and $Q^{r,n}_{\textbf{q},\sigma}(R):=R_{\textbf{q},\sigma}[x_1^{\pm 1
},\dots,x_r^{\pm 1}, x_{r+1},\dots,x_n]$. Then,
\begin{enumerate}
\item[\rm (i)]\
\begin{center}
$\text{\rm lgld}(R)\leq \text{\rm lgld}(Q^{r,n}_{\textbf{q},\sigma}(R))\leq \text{\rm lgld}(R)+n$,
if $\text{\rm lgld}(R)<\infty$,
\end{center}
\begin{center}
$\text{\rm lgld}(R_{\textbf{q}}[x_1,\dots,x_n])=\text{\rm lgld}(R)+n$.
\end{center}
If $R$ is semisimple, then $\text{\rm lgld}(R_{\textbf{q}}[x_1,\dots,x_n])=n$ and $\text{\rm
lgld}(R[x^{\pm 1};\sigma])=1$.
\item[\rm (ii)]If $R$ is a left Noetherian, then
\begin{center}
${\rm lKdim}(R)\leq {\rm lKdim}(Q^{r,n}_{\textbf{q},\sigma}(R))\leq {\rm lKdim}(R)+n$,
\end{center}
\begin{center}
$\text{\rm lKdim}(R_{\textbf{q}}[x_1,\dots,x_n])=\text{\rm lKdim}(R)+n$.
\end{center}
If  $R=\Bbbk$ is a field,
\begin{center}
$\text{\rm lKdim}(\Bbbk_{\textbf{q}}[x_1,\dots,x_n])=n$, $\text{\rm
lKdim}(\Bbbk_{\textbf{q}}[x_1^{\pm 1},\dots,x_n^{\pm 1}])\leq n$ and $\text{\rm lKdim}(\Bbbk[x^{\pm
1};\sigma])=1$.
\end{center}
\item[\rm (iii)]If $R$ is a left Noetherian domain, then ${\rm
ludim}(Q^{r,n}_{\textbf{q},\sigma}(R))=1$.
\end{enumerate}
\end{theorem}
\begin{proof}
(i) $Q^{r,n}_{\textbf{q},\sigma}(R)$ is left free over $R$, but $Q^{r,n}_{\textbf{q},\sigma}(R)$ is
right free over $R$, also with basis (\ref{equ1.4.2}):
$rx^{\alpha}=x^{\alpha}(\sigma^{\alpha})^{-1}(r)$ and $x^{\alpha}r=\sigma^{\alpha}(r)x^{\alpha}$;
note that in $Q^{r,n}_{\textbf{q},\sigma}(R)$ holds the identity
$x_i^{-1}r=\sigma_i^{-1}(r)x_i^{-1}$, for $1\leq i\leq r$. Then the result follows from
\cite{McConnell} Theorem 7.2.6 and Corollary 7.4.3, and Theorem \ref{3.1.1a}.

From Theorem \ref{3.1.1a} we get also that $\text{\rm
lgld}(R_{\textbf{q}}[x_1,\dots,x_n])=\text{\rm lgld}(R)+n$, but if $R$ is semisimple, then
$\text{\rm lgld}(R_{\textbf{q}}[x_1,\dots,x_n])=n$.

For $R$ semisimple, $0\leq \text{\rm lgld}(R[x^{\pm 1};\sigma])\leq 1$, but $R[x^{\pm 1};\sigma]$
is not left Artinian, so $\text{\rm lgld}(R[x^{\pm 1};\sigma])=1$. Note that this result coincides
with \cite{McConnell}, Theorem 7.5.3.

(ii) As we saw above $Q^{r,n}_{\textbf{q},\sigma}(R)$ is left and right free over $R$, then the
result follows from Corollary \ref{1.3.4}, \cite{McConnell} Corollary 6.5.3 and Lemma 6.5.3,
\cite{Goodearl} Exercise 15U, and Theorem \ref{3.1.1a}.

If $\Bbbk$ is a field, $0\leq \text{\rm lKdim}(\Bbbk[x^{\pm 1};\sigma])\leq 1$, but $\Bbbk[x^{\pm
1};\sigma]$ is not left Artinian, so $\text{\rm lKdim}(\Bbbk[x^{\pm 1};\sigma])=1$. Compare these
results with \cite{McConnell}, Proposition 6.5.4, \cite{Goodearl}, Corollary 15.20 and
\cite{Artamonov2}, Theorem 3.1.4.

(iii) This follows from \cite{McConnell} Lemma 2.2.12, and Theorem \ref{3.1.1a}.
\end{proof}
\begin{remark}
The estimations of dimensions that we presented above agree with some exactly computations that we
found in the literature, see \cite{Artamonov}, \cite{Artamonov2}, \cite{Bavulakrull},
\cite{Bavula8}, \cite{Bavuladown}, \cite{Fujita}, \cite{Giaquinto}, \cite{Kirkman},
\cite{smithsl2}.
\end{remark}
\section{K-theory}

In this section we compute Quillen's $K$-groups for bijective skew \textit{PBW} extensions, in
particular, we compute Grothendieck, Bass and Milnor's groups for these extensions. For this we use
the following deep result due Quillen (\cite{QuillenK}, see also \cite{Hodges}): \textit{Let $B$ be
a filtered ring with filtration $\{B_p\}_{p\ge 0}$ such that $B_0$ and $Gr(B)$ are left Noetherian
left regular rings. Then the functor $B\otimes_{B_0}\_\_$ induces isomorphisms
\[
K_i(B)\cong K_i(B_0),\ \ \ {\rm for\ all}\ \ i\ge 0.
\]}
\begin{theorem}\label{K_iforskpw}
Let $R$ be a left Noetherian left regular ring. If $A$ is a bijective skew \textit{PBW} extension
of $R$, then $K_i(A)\cong K_i(R)$ for all $i\ge 0$.
\begin{proof}
Theorem \ref{1.3.2} describes explicitly the filtration of $A$ with $F_0=R$. The proof of Corollary
\ref{1.3.4} shows that $Gr(A)$ is left Noetherian and the proof of Corollary \ref{3.2.1a}
guarantees that $Gr(A)$ is left regular. The assertion follows from Quillen's result.
\end{proof}
\end{theorem}
According to Theorem \ref{K_iforskpw}, to effective calculate Quillen's $K$-groups for bijective
skew \textit{PBW} extensions we have to compute $K_i(R)$, but for many of remarkable examples of
bijective skew \textit{PBW} extensions the ring $R$ of coefficients is an iterated Laurent
polynomial ring. Thus, Theorem \ref{K_iforskpw} can be complemented with the next proposition that
can be used to effective compute Quillen's $K$-groups for the examples considered in Section
\ref{Examplestrad}.
\begin{proposition}\label{K_1loc}
Let $B$ be a left Noetherian left regular ring. Then
\begin{equation}\label{equ2.1}
K_m(B[x^{\pm 1}_1,\dotsc,x^{\pm 1}_n])\cong \bigoplus_{j=0}^m [K_j(B)]^{_n{\rm C}_{m-j}},\ \ \ {\rm
where}\ \ \ _n{\rm C}_{m-j}:=\dbinom{n}{m-j}.
\end{equation}
In particular, Grothendieck, Bass and Milnor's groups of $B[x^{\pm 1}_1,\dotsc,x^{\pm 1}_n]$ are
given by
\begin{enumerate}
\item [\rm (i)] $K_0(B[x^{\pm 1}_1,\dotsc,x^{\pm 1}_n])\cong K_0(B)${\rm ;}
\item [\rm (ii)] $K_1(B[x^{\pm 1}_1,\dotsc,x^{\pm 1}_n])\cong [K_0(B)]^{n}\oplus K_1(B)${\rm ;}
\item [\rm (iii)] $K_2(B[x^{\pm 1}_1,\dotsc,x^{\pm 1}_n])\cong [K_0(B)]^{\frac{n(n-1)}{2}}\oplus [K_1(B)]^{n}\oplus K_2(B)${\rm .}
\end{enumerate}
\end{proposition}
\begin{proof}
With the conditions on $B$, the proof is by induction on $m$ and $n$ and using the following well
known identity (\cite{QuillenK} Theorem 8 or \cite{Srinivas} Corollary 5.5):
\begin{equation}\label{equ2.2}
K_i(B[x^{\pm 1}])\cong K_i(B)\oplus K_{i-1}(B), \ \ {\rm where} \ \ {K_{-1}(B):=0}.
\end{equation}
\end{proof}
Next we compute the Quillen's $K$-groups for skew quantum polynomials, we will see that this
computations generalize the results presented in \cite{Artamonov2}.
\begin{lemma}\label{3.3}
If $R$ is a left Noetherian left regular ring, then
\begin{equation}\label{Kcomputimpor}
K_i(R_{\textbf{q},\sigma}[x_1^{\pm 1},\dotsc,x_{r}^{\pm 1}, x_{r+1},\dotsc, x_n])\cong
K_i(R_{\textbf{q},\sigma}[x_1^{\pm 1},\dotsc,x_r^{\pm 1}]), i\geq 0.
\end{equation}
\end{lemma}
\begin{proof}
Note that $R_{\textbf{q},\sigma}[x_1^{\pm 1},\dotsc,x_{r}^{\pm 1}, x_{r+1},\dotsc,x_n]$ is as a
quasi-commutative bijective skew $PBW$ extension of the $r$-multiparametric skew quantum torus over
$R$, i.e.,
\begin{center}
$R_{\textbf{q},\sigma}[x_1^{\pm 1},\dotsc,x_{r}^{\pm 1}, x_{r+1},\dotsc,x_n]\cong \sigma(T)\langle
x_{r+1},\dots,x_{n}\rangle$, with $T:=R_{\textbf{q},\sigma}[x_1^{\pm 1},\dotsc,x_{r}^{\pm 1}]$.
\end{center}
Then the result follows from Theorem \ref{K_iforskpw} since $T$ is left Noetherian left regular
(see Remark \ref{3.3a}).
\end{proof}
This lemma says that computing Quillen's $K$-groups is reduced to compute
$K_i(R_{\textbf{q},\sigma}[x_1^{\pm 1},\dotsc,x_r^{\pm 1}])$ for $i\ge 0$, but note that
\[
R_{\textbf{q},\sigma}[x_1^{\pm 1},\dotsc,x_r^{\pm 1}] \cong R[x_1^{\pm 1};\sigma_1]\dotsb
[x_{r}^{\pm 1};\sigma_{r}],\ \
x_jr=\sigma_j(r)x_j,\ \ 1\le j\le r, \ \sigma_j(x_i):=q_{ij}x_i,\\
\ 1\le i<j\le r,
\]
so
\begin{equation}\label{15ene}
K_i(R_{\textbf{q},\sigma}[x_1^{\pm 1},\dotsc,x_r^{\pm 1}])\cong K_i(R[x_1^{\pm 1};\sigma_1]\dotsb
[x_{r}^{\pm 1};\sigma_{r}]),\ \ \ i\ge 0.
\end{equation}
In other words, to calculate Quillen's $K$-groups we need to compute $K_i(R[x_1^{\pm
1};\sigma_1]\dotsb [x_{r}^{\pm 1};\sigma_{r}]),\ i\ge 0$. With this purpose in mind, we present the
following proposition which is a generalization of (\ref{equ2.2}) for skew Laurent polynomial
rings.
\begin{proposition}\label{generFtrr}
Let $B$ be a left Noetherian left regular ring. If $\sigma$ is an automorphism of $B$ that acts
trivially on the $K$-theory of $B$, then
\[
K_i(B[x^{\pm 1};\sigma])\cong K_i(B)\oplus K_{i-1}(B),\ \ {\rm where}\ \ K_{-1}(B):=0.
\]
\begin{proof}
The idea is to apply the following long exact sequence of $K$-groups (see \cite{Yao} Corollary
2.2.):
\begin{equation}\label{Quillenexact}
\dotsb \to K_i(B)\xrightarrow{1-\sigma_{*}} {K_i(B)} \to K_i(B[x^{\pm 1};\sigma])\to K_{i-1}(B)\to
\dotsb.
\end{equation}
The assumption $\sigma$ acts trivially on the $K$-theory of $B$ implies that from the long exact
sequence (\ref{Quillenexact}) we can extract a short exact sequence
\[
0\to K_i(B)\to K_i(B[x^{\pm 1};\sigma])\to K_{i-1}(B)\to 0
\]
(c.f. \cite{QuillenK}, \cite{Srinivas} or \cite{Yao}), and hence $K_i(B[x^{\pm 1};\sigma])\cong
K_i(B)\oplus K_{i-1}(B)$.
\end{proof}
\end{proposition}
\begin{proposition}\label{Kprel}
Suppose that $R$ is left Noetherian left regular and that $\sigma_j,\ 1\le j\le r-1$, acts
trivially on the $K$-theory of the $j-1$-multiparametric skew quantum torus $R[x_1^{\pm
1};\sigma_1]\dotsb [x_{j-1}^{\pm 1};\sigma_{j-1}]$. Then
\begin{equation}
K_m(R[x_1^{\pm 1};\sigma_1]\dotsb [x_{r}^{\pm 1};\sigma_{r}])\cong \bigoplus_{j=0}^m
[K_j(R)]^{_r{\rm C}_{m-j}}.
\end{equation}
\begin{proof}
Similar to the proof of Proposition \ref{K_1loc} and using Proposition \ref{generFtrr}.
\end{proof}
\end{proposition}

\begin{proposition}\label{sqtorus}
If $R$ is left Noetherian left regular, then for the $r$-multiparametric skew quantum torus
$R_{{\rm q},\sigma}[x_1^{\pm 1},\dotsc,x_{r}^{\pm 1}]$ we have
\[
K_1(R_{{\rm q},\sigma}[x_1^{\pm 1},\dotsc,x_{r}^{\pm 1}])\cong [K_0(R)]^r\oplus K_1(R).
\]
\begin{proof}
Similar to the proof of Proposition \ref{K_1loc} and using the fact that if $B$ is a left
Noetherian left regular ring with an automorphism $\sigma$ and $i:B\to B[x^{\pm 1};\sigma]$ is the
natural embedding, then the sequence
\[
1\to K_1(B)\xrightarrow{i_{*}} K_1(B[x^{\pm 1};\sigma])\xrightarrow{\partial} K_0(B)\to 0
\]
is exact (see \cite{QuillenK}, p. 122 and \cite{Bass}, Ch. 9. Theorem 6.3).
\end{proof}
\end{proposition}
\begin{corollary}\label{ArtBassQ}
If $R$ is left Noetherian left regular then for the skew quantum polynomial we have
\[
K_1(R_{{\rm q},\sigma}[x_1^{\pm 1}, \dotsc, x_r^{\pm 1},x_{r+1},\dotsc,x_n])\cong [K_0(R)]^r\oplus
K_1(R).
\]
\begin{proof}
Follows from Lemma \ref{3.3} and Proposition \ref{sqtorus}.
\end{proof}
\end{corollary}
Now we can establish the main result in this section.
\begin{theorem}\label{Ktheoryspbw}
Under the same conditions of Proposition \ref{Kprel}, we have
\[
K_m(R_{{\rm q},\sigma}[x_1^{\pm 1}, \dotsc, x_r^{\pm 1},x_{r+1},\dotsc,x_n])\cong \bigoplus_{j=0}^m
[K_j(R)]^{_r{\rm C}_{m-j}}.
\]
In particular,
\begin{align*}
K_0(R_{{\rm q},\sigma}[x_1^{\pm 1}, \dotsc, x_r^{\pm 1},x_{r+1},\dotsc,x_n])& \cong K_0(R);\\
K_1(R_{{\rm q},\sigma}[x_1^{\pm 1}, \dotsc, x_r^{\pm 1},x_{r+1},\dotsc,x_n])& \cong [K_0(R)]^r\oplus K_1(R);\\
K_2(R_{{\rm q},\sigma}[x_1^{\pm 1}, \dotsc, x_r^{\pm 1},x_{r+1},\dotsc,x_n])& \cong
[K_0(R)]^{\frac{r(r-1)}{2}}\oplus [K_1(R)]^r\oplus K_2(R).
\end{align*}
\begin{proof}
Follows from Lemma \ref{3.3} and Proposition \ref{Kprel}.
\end{proof}
\end{theorem}

\smallskip

\begin{center}
\textbf{Acknowledgements}
\end{center}
The authors are grateful to the editors and the referee for valuable suggestions and corrections.


\end{document}